\documentclass{amsart}
\usepackage{amssymb}
\usepackage{latexsym}
\usepackage{amsmath}

\def\wt{\widetilde}
\def\wh{\widehat}
\def\ov{\overline}
\def \im{{\rm Im}}
 \def\up{\upharpoonright}
\def\cH{\mathcal H} \def\cB{\mathcal B}

\def\cK{\mathcal K} \def\cL{\mathcal L}
 \def\cN{\mathcal N} \def\cP{\mathcal P} 
 \def\cT{\mathcal T} \def\cI{\mathcal I}

\def \gH{\mathfrak H}   \def \gN{\mathfrak N}

\def \bC{\mathbb C}    \def\bR{\mathbb R}
\def\bH{\mathbb H} 
\def \l{\lambda}
\def \a{\alpha} \def \b{\beta}  \def \L{\Lambda}  \def \s{\sigma}
 \def \t{\theta} \def\g {\gamma}
\def\d {\delta}  \def\om {\omega} \def\Om {\Omega} 
\def \f{\varphi}  \def \G{\Gamma} \def\D {\Delta} \def\Si{\Sigma}

\def \C{\widetilde {\mathcal C}}
\def \CA{\C(\cH_0,\cH_1)}

\def \cd {\cdot}

\def\AC {AC(\cI; \bH)}   \def\LI {L_\Delta^2(\cI)}
\def\lI {\cL_\Delta^2(\cI)}
\def\LS {L^2(\Sigma ;\bH)} \def\lS {\cL^2(\Sigma; \bH)}

\def\tma{\cT_{\max}} \def\tmi{\cT_{\min}} \def\Tma{T_{\max}} \def\Tmi{T_{\min}}
\def\Sel{\wt{\rm Self}_0 (\Tmi)} \def\Selo{{\rm Self}_0 (\Tmi) }

\def \dom {{\rm dom}\,}  \def \ran {{\rm ran}\,}  \def \ker{{\rm ker\,}}
 \def \mul {{\rm mul}\,}

  \def\tm{\times}

\def \pair {\tau=\{\tau_+,\tau_-\}}
\def \CR {\bC\setminus\bR}

\newcommand {\lo}[1] {\cL_\D^2[#1,\bH ]}

\def\bt{\{\cH,\G_0,\G_1\}}
\def\bta{\{\cH_0\oplus \cH_1,\Gamma _0,\Gamma _1\}}

\newtheorem{theorem}{Theorem}[section]
\newtheorem{proposition}[theorem]{Proposition}
\newtheorem{corollary}[theorem]{Corollary}
\newtheorem{lemma}[theorem]{Lemma}
\newtheorem{assertion}[theorem]{Assertion}
\theoremstyle{definition}

\theoremstyle{definition}
\newtheorem {definition} [theorem]{Definition}
\theoremstyle{remark}
\newtheorem{remark}[theorem]{Remark}
\numberwithin{equation}{section}
\begin{document}
\title[On spectral and pseudospectral functions]
{On spectral and pseudospectral functions of first-order symmetric systems }
\author {Vadim Mogilevskii}
\address{Institute of Applied Mathematics and Mechanics, NAS of
Ukraine,  R. Luxemburg Str. 74,  83050 Donetsk,    Ukraine}
\address{Department of Mathematics,  Lugans'k National University,
 Oboronna Str. 2, 91011  Lugans'k,   Ukraine}
\email{vim@mail.dsip.net}
\begin{abstract}
We consider general (not necessarily Hamiltonian) first-order symmetric system $J y'-B(t)y=\D(t) f(t)$ on an interval $\cI=[a,b) $ with the regular
endpoint $a$. A distribution matrix-valued function $\Si(s), \; s\in\bR,$ is called a spectral (pseudospectral) function of such  a system if the corresponding Fourier transform is an isometry (resp. partial isometry) from $\LI$ into $L^2(\Si)$. The main result is a parametrization of all spectral and pseudospectral functions of a given system  by
means of a Nevanlinna boundary parameter  $\tau$. Similar
parameterizations for various classes of boundary  problems
have earlier been obtained by Kac and Krein, Fulton, Langer and Textorius, Sakhnovich and others.
\end{abstract}
\maketitle
\section{Introduction}
Let  $H$ and $\wh H$ be finite dimensional Hilbert spaces and let
\begin {equation}\label{1.1}
\bH:=H\oplus\wh H \oplus H.
\end{equation}
 Denote also by $[\bH]$ the set of all linear
operators in $\bH$. We study first-order symmetric systems  of differential equations defined on an interval $\cI=[a,b), -\infty<a <b\leq\infty,$ with the
regular endpoint $a$ and regular or singular endpoint $b$. Such a system is of the form \cite{Atk,GK}
\begin {equation}\label{1.2}
J y'-B(t)y=\D(t) f(t), \quad t\in\cI,
\end{equation}
where $B(t)=B^*(t)$ and $\D(t)\geq 0$ are locally integrable $[\bH]$-valued functions on $\cI$ and
\begin {equation} \label{1.3}
J=\begin{pmatrix} 0 & 0&-I_H \cr 0& i I_{\wh H}&0\cr I_H&
0&0\end{pmatrix}:H\oplus\wh H\oplus H \to H\oplus\wh H\oplus H.
\end{equation}
 Recall that system \eqref{1.2} is called definite if  each  solution of the homogeneous system
\begin {equation}\label{1.5}
J y'-B(t)y=\l \D(t) y, \quad \l \in\bC
\end{equation}
satisfying $\D(t)y(t)=0$ (a.e. on $\cI$) is trivial, i.e., $y(t)=0,\; t\in\cI$. Recall also that system \eqref{1.2} is called a Hamiltonian system if $\wh H=\{0\}$ and hence
\begin {equation}\label{1.6}
J=\begin{pmatrix}  0&-I_H \cr  I_H& 0\end{pmatrix}:H\oplus H \to H\oplus H.
\end{equation}
Let $\gH:=\LI$ be the Hilbert space of  functions $f(\cd):\cI\to\bH$  satisfying $ \int\limits_\cI (\D(t)f(t),f(t))_\bH\,dt <\infty$ and let $\gH_b$ be the set of  functions $f\in\gH$ with compact support. Denote also by $Y_0(\cd,\l)$ the $[\bH]$-valued operator solution of \eqref{1.5} satisfying $Y_0(a,\l)=I_\bH$. For each function $f(\cd)\in\gH_b$ define the Fourier transform
\begin {equation}\label{1.6a}
\wh f(s)=\int_\cI Y_0^*(t,s)\D(t)f(t)\,dt.
\end{equation}
\begin{definition}\label{def1.1}
A distribution $[\bH]$-valued function $\Si(s),\; s\in\bR,$ is called a spectral function of the system \eqref{1.2} if for each $f\in\gH_b$ the following Parseval equality holds
\begin {equation}\label{1.6b}
\int_\cI (\D(t)f(t),f(t))\, dt=\int_\bR (d\Si(s)\wh f(s),\wh f(s)).
\end{equation}
This means that the mapping $Vf=\wh f,\; f\in\gH_b,$ admits an extension to an isometry $V=V_\Si$ from $\gH$ to $\LS$ (for definition of the Hilbert space $\LS$ see \cite{DunSch,Kac50} and also Section \ref{sub2.3}).
\end{definition}
As is known the extension theory of symmetric linear relations gives a natural framework   for studying of spectral functions of symmetric systems.Assume that system \eqref{1.2} is definite. Then according to \cite{Kac03,LesMal03, Orc}   this system  generates the minimal linear relation $\Tmi$ and the maximal
linear relation $\Tma$ in   $\gH$.  It turns out that $\Tmi$ is a closed symmetric relation with possibly nontrivial multivalued part $\mul \Tmi=\{f\in\gH: \{0,f\}\in \Tmi\}$. Moreover, the  deficiency indices $n_\pm(\Tmi)$ of $\Tmi$ are not necessarily equal and satisfy $n_\pm(\Tmi)\leq\dim\bH$.

Recall that a linear relation $\wt T=\wt T^*\supset\Tmi$ in a Hilbert space $\wt\gH\supset \gH$ is called an exit space self-adjoint extension of $\Tmi$. Denote by $\wt {\rm Self} (\Tmi)$ the set of all (minimal) exit space self-adjoint extensions of $\Tmi$ and let $\Sel $ be the set of all $\wt T\in \wt {\rm Self} (\Tmi)$ with $\mul\wt T=\mul\Tmi$. Each extension $\wt T\in \wt {\rm Self} (\Tmi)$ generates a generalized resolvent $R(\l)$ of $\Tmi$ defined by
\begin {equation}\label{1.7a}
R(\l)=P_\gH(\wt T-\l)^{-1}\up\gH,\quad \l\in\CR.
\end{equation}
According to \cite{Bru78,DLS88,Sht57}  $R(\l)$ admits the representation
\begin {equation}\label{1.8}
(R(\l) f)(x)=\int_\cI Y_0(x,\l)(\Om(\l)+\tfrac 1 2 \, {\rm
sgn}(t-x)J)Y_0^*(t,\ov\l)\D(t) f(t)\,dt , \quad f\in\gH,
\end{equation}
where $\Om(\cd)=\Om_{\wt T}(\cd):\CR\to [\bH]$ is an operator function called a characteristic matrix of the system \eqref{1.2}. Since $\Om(\cd)$ is a Nevanlinna function, the equality (the Stieltjes formula)
\begin {equation}\label{1.9}
\Si(s)=\Si_\Om(s)=\lim\limits_{\d\to+0}\lim\limits_{\varepsilon\to +0} \frac 1 \pi\int_{-\d}^{s-\d}\im \,\Om(\s+i\varepsilon)\, d\s
\end{equation}
defines a distribution $[\bH]$-valued function $\Si_\Om(\cd)$ (the spectral function of $\Om(\cd)$).

Let $X_\D=\{t\in\cI: \D(t)\;\; \text{is invertible}\}$ and let $\mu_1$ be the Lebesgue measure on $\cI$. If $\mu_1(\cI\setminus X_\D)=0$, then $\Tmi$ is a densely defined operator in $\gH$ and hence each extension $\wt T\in \wt {\rm Self} (\Tmi)$ is an operator (this implies that $\wt {\rm Self} (\Tmi)=\Sel$). By using the same methods as in \cite{Sht57} one can show that in this case for each $\wt T\in \wt {\rm Self} (\Tmi)$ the equalities \eqref{1.7a}--\eqref{1.9} define a spectral function $\Si(\cd)$ of the system \eqref{1.2}.

If $\mu_1(\cI\setminus X_\D)>0$, then  the situation is more complicated. In particular, in this case spectral functions of the system \eqref{1.2} may not exist.

Spectral type functions of Hamiltonian systems \eqref{1.2} with $\mu_1(\cI\setminus X_\D)\geq 0$ were studied in \cite{Kac83,Kac03,Sah90,Sah13}. Namely, let $\f(t,\l)=(\f_1(t,\l),\, \f_2(t,\l))^\top(\in [H,H\oplus H])$ be an operator solution of \eqref{1.5} with $\f_1(a,\l)=I_H$ and $\f_2(a,\l)=0$ and let
\begin {equation}\label{1.12}
\wh f_\f(s)=\int_\cI \f^*(t,s)\D(t)f(t)\, dt, \quad s\in \bR
\end{equation}
be the "truncated" Fourier transform of a function $f\in\gH_b$ (cf. \eqref{1.6a}). Moreover, let $T$ be a symmetric extension of $\Tmi$ defined as a closure of the set of all $\{y,f\}\in\Tma$ such that a function $y=\{y_1(t), y_2(t)\}(\in H\oplus H)$ has compact support and satisfies $y_2(a)=0$. Assume also that $\mul T=\{f\in\gH:\{0,f\}\in T\}$ is a multivalued part of $T$ and let  $\gH_0:=\gH\ominus \mul T$, so that
\begin {equation}\label{1.13}
\gH=\mul T\oplus \gH_0.
\end{equation}
In the papers by Kats \cite{Kac83,Kac03} Hamiltonian systems \eqref{1.2} with $H=\bC$   and $B(t)\equiv 0$ were considered. In these papers a quasispectral function is defined as a scalar distribution function $\s_\f (s), \; s\in\bR,$ satisfying
\begin {equation}\label{1.14}
\int_\cI (\D(t)f(t), f(t))\,dt=\int_\bR |\wh f_\f(s)|^2\, d\s_\f (s)
\end{equation}
for all $f\in\gH_b\cap\gH_0$. Moreover, $\s_\f (s)$ is called a spectral function if \eqref{1.14} holds for all $f\in\gH_b$.
It  is shown in \cite{Kac03} that at least one quasispectral function always exists, while a spectral function exists if and only if $\mul T=\{0\}$. Observe also that the subspace $\gH_0$ is characterized in \cite{Kac03} in terms of indivisible intervals for $\D(t)$.

Recall that system \eqref{1.2} is called regular if $b<\infty$ and the coefficients $\D(t)$ and $B(t)$ are integrable on $\cI=[a,b]$. Spectral and pseudospectral functions of regular Hamiltonian systems with $B(t)\equiv 0$ were studied in \cite{Sah90} (see also \cite{Sah13}). Clearly, for such a system the Fourier transform \eqref{1.12} is well defined for every function $f\in\gH$; moreover, \cite[Lemma 6]{Sah90} yields $\mul T=\{f\in\gH:\wh f_\f(s)\equiv 0,\; s\in\bR\}$. According to \cite{Sah90} a distribution $[H]$-valued function $\Si_\f(\cd)$ is called a pseudospectral function if
\begin {equation*}
\int_\cI (\D(t)f(t), f(t))\,dt=\int_\bR ( d\Si_\f (s)\wh f_\f(s), \wh f_\f(s)), \quad f\in\gH_0.
\end{equation*}
Under certain additional assumptions, a description of all  pseudospectral (and spectral) functions $\Si_\f(\cd)$ of a regular Hamiltonian system \eqref{1.2} is obtained in \cite{Sah90,Sah13}. Such a description is given in terms of a   linear-fractional transform of a Nevanlinna operator pair which plays a role of a parameter.

Spectral functions of a general (not necessarily Hamiltonian) definite system \eqref{1.2} defined on  $\cI=[a,b), \; b\leq \infty,$ were studied in  \cite{LanTex78,LanTex84,LanTex85}. According to \cite{LanTex84} an  $[\bH]$-valued distribution function $\Si(\cd)$ is called a spectral function of system \eqref{1.2} if the mapping $Vf=\wh f$  defined by \eqref{1.6a}  admits an extension to a contraction  $V$ from $\gH$ to $\LS$ satisfying $||Vf||=||f||$ for all $f\in\dom \Tma$. By using the Krein's method of directing mappings the authors establish in \cite{LanTex84} a correspondence between spectral functions and self-adjoint extensions of $\Tmi$.

In the present paper we study spectral and pseudospectral functions of definite symmetric systems \eqref{1.2}. We specify a connection between boundary problems for such systems and pseudospectral functions $\Si(\cd)$. This enables us to parameterize all functions $\Si(\cd)$ in terms of a boundary parameter.

Note that some of our results are closely related to those from \cite{LanTex84,LanTex85} (for more details see Remarks \ref{rem4.5a}, \ref{rem4.13a} and  \ref{rem4.19} below).

Assume that $\Si(\cd)$ is an $[\bH]$-valued distribution function on $\bR$ satisfying the following condition:

(C) The mapping $Vf=\wh f$ originally defined by \eqref{1.6a} on $\gH_b$ admits an extension to a partial isometry $V=V_\Si$ from $\gH$ to $\LS$.

We prove that in this case
\begin {equation}\label{1.15}
\mul\Tmi \subset \ker V_\Si.
\end{equation}

For an $[\bH]$-valued distribution function $\Si(\cd)$ satisfying the condition (C) and the additional condition $||V_\Si f||=||f||, \; f\in\dom\Tmi,$ the inclusion \eqref{1.15} can be derived from the results of \cite{LanTex84} (see Remark \ref{rem4.5a} below).

Let system \eqref{1.2} be Hamiltonian with $\dim H=1$ and  $B(t)\equiv 0$, let $\wh f_\f(s)$ be the Fourier transform \eqref{1.12} and let $\s (s)$ be a scalar distribution function such that the mapping $Vf=\wh f,\; f\in \gH_b,$ admits an extension to a partial isometry $V$ from $\gH$ to $L_\s^2$. Then   the inclusion $\mul T \subset \ker V$ follows from the proof of Lemma 4.1 in \cite{Kac03}. Note that this proof is based on the method of indivisible intervals, which is not elaborated for the case $\dim H>1$.

The inclusion \eqref{1.15} makes natural the following definition (cf. Definition \ref{def1.1}).
\begin{definition}\label{def1.2}
An $[\bH]$-valued distribution function $\Si(\cd)$ on $\bR$ satisfying the condition (C) is called a pseudospectral function of the system \eqref{1.2} if $\ker V_\Si=\mul\Tmi$.
\end{definition}
Clearly, a pseudospectral function $\Si(\cd)$ is a spectral function in the sense of \cite{LanTex84}.

The main result of the paper is a parametrization of  all pseudospectral functions in terms of a Nevanlinna boundary parameter. Such a parametrization is obtained for  absolutely definite systems satisfying $n_-(\Tmi)\leq n_+(\Tmi)$ (system \eqref{1.2} is called absolutely definite if $\mu_1 (X_\D)>0$). However to simplify presentation we additionally assume below (within this section) that system is Hamiltonian and $n_-(\Tmi)= n_+(\Tmi)$.  In this case there exist a finite dimensional  Hilbert  space $\cH_b$
 and a surjective linear mapping $
\G_b=(\G_{0b}, \G_{1b})^\top:\dom\Tma\to
 \cH_b \oplus \cH_b$
such that
\begin {equation*}
 [y,z]_b(=\lim\limits_{t\uparrow b} (J y(t),z(t)))= (\G_{0b}y,\G_{1b}z)-(\G_{1b}y,\G_{0b}z),\quad y,z \in \dom\Tma
\end{equation*}
In fact $\G_b y$ is a singular  boundary value of a function $y\in\dom\Tma$ (for more details see Remark 3.5 in \cite{AlbMalMog13}).

Assume that $\cH_b$ and $\G_b$ are fixed. For a function $y\in\dom\Tma$  let
$\G_0'y=\{ - y_1(a),\;\G_{0b}y\}\in  H\oplus \cH_b$ and $\G_1'y= \{y_0(a),-\G_{1b}y\}\in H\oplus \cH_b$, where $y_0(a)$ and $y_1(a)$ are taken from the representation $y(t)=\{y_0(t),\,y_1(t)\}(\in H\oplus H)$ of $y$. We show that for each generalized resolvent $R(\l)$ of $\Tmi$ there exists a unique Nevanlinna pair  $\tau=\{C_0(\l),C_1(\l)\}$ of operator functions $C_0(\l), \; C_1(\l)(\in [H\oplus\cH_b]),\; \l\in\CR,$ such that a function $y(t)=(R(\l)f)(t), \; f=f(\cd)\in\gH,$ is an $L_\D^2$-solution of the following boundary problem:
\begin {gather}
J y'-B(t)y = \l \D(t) y+\D(t)f(t), \quad t\in\cI\label{1.16}\\
C_0(\l)\G_0'y-C_1(\l)\G_1'y=0,\quad \l\in\CR.\label{1.17}
\end{gather}
Note, that  \eqref{1.17} is a boundary condition imposed on boundary values of a function $y\in\dom\Tma$. One may consider a Nevanlinna pair $\tau$
as a  boundary parameter, since $R(\l)$ runs over the set of all generalized resolvents of $\Tmi$ when $\tau$ runs over the set of all
Nevanlinna pairs $\tau=\{C_0(\l),C_1(\l)\}$. To indicate this fact explicitly we write $R(\l)=R_\tau(\l)$ and $\Om(\l)=\Om_\tau(\l)$ for the generalized resolvent of $\Tmi$ and the corresponding characteristic matrix respectively.

The main result can be formulated in the form of the following theorem.
\begin{theorem}\label{th1.3}
If system \eqref{1.2} is absolutely definite, then there exist operator functions $\Om_0(\l)(\in [\bH]), \; S(\l)(\in [H\oplus\cH_b,\bH])$ and a Nevanlinna operator function $M(\l)(\in [H\oplus\cH_b]), \; \l\in\CR,$ such that the equality
\begin {equation}\label{1.18}
\Om(\l)=\Om_\tau(\l)=\Om_0(\l)+S(\l)(C_0(\l)-C_1(\l)M(\l))^{-1}C_1(\l)S^*(\ov\l), \quad \l\in\CR
\end{equation}
together with the Stieltjes formula \eqref{1.9} establishes a bijective correspondence between all  boundary parameters $\tau=\{C_0(\l),C_1(\l)\}$ satisfying
\begin {equation*}
\lim_{y\to \infty} \tfrac 1 {i y} (C_0(i y)-C_1(i y )M(i y))^{-1}C_1(i y) =\lim_{y\to \infty} \tfrac 1 {i y} M(i y)( C_0(i y)-C_1(i y )M(i y) )^{-1} C_0(i y)=0
\end{equation*}
and all pseudospectral functions  $\Si(\cd)=\Si_\tau(\cd)$ of the system.
\end{theorem}
Note that the operator functions $\Om_0(\cd), \; S(\cd)$ and $M(\cd)$ in \eqref{1.18} are defined in terms of the boundary values of respective $L_\D^2$-operator solutions of Eq.  \eqref{1.5}. Observe also that in the case of maximal deficiency indices a description  of spectral and pseudospectral functions for certain  classes of boundary value problems in the form close to \eqref{1.18}, \eqref{1.9} has been obtained in \cite{Ful77,Gor66,HinSha82,Hol85,KacKre,Mog07}(for regular symmetric systems see \cite{LanTex84,LanTex85} and \cite{Sah90}). Moreover,
similar to \eqref{1.18}, \eqref{1.9} parametrization of $[H\oplus\wh H]$-valued pseudospectral functions $\Si(\cd)$ of a singular system \eqref{1.2} with arbitrary deficiency indices of $\Tmi$
 can be found in  recent works \cite{AlbMalMog13,Mog13.1}.

It follows from \eqref{1.15} that the set of spectral functions of the system \eqref{1.2} is not empty if and only if $\mul\Tmi=\{0\}$. Moreover, if this condition is satisfied, then the set of spectral functions coincides with the set of pseudospectral functions and, consequently, all the above results hold for spectral functions.

In conclusion note that, in the case of a definite regular Hamiltonian system \eqref{1.2}, for each pseudospectral function   $\Si_\f(s)(\in [H])$ in the sense of \cite{Sah90,Sah13} there is a pseudospectral function $\Si(s)(\in [H\oplus H])$ in the sense of Definition \ref{def1.2} corresponding to appropriate  separated boundary conditions \eqref{1.17} and such that $\Si(s)= {\rm diag} (\Si_\f(s), \, 0)$. Hence the results from \cite{Sah90,Sah13} concerning pseudospectral functions of definite systems can be developed by using the results of the present paper (this assertion will be clarified in more details elsewhere).
\section{Preliminaries}
\subsection{Notations}
The following notations will be used throughout the paper: $\gH$, $\cH$ denote Hilbert spaces; $[\cH_1,\cH_2]$  is the set of all bounded linear operators defined on the Hilbert space $\cH_1$ with values in the Hilbert space $\cH_2$; $[\cH]:=[\cH,\cH]$; $P_\cL$ is the orthoprojection in $\gH$ onto the subspace $\cL\subset\gH$; $\bC_+\,(\bC_-)$ is the upper (lower) half-plane  of the complex plane.

Recall that a closed linear   relation   from $\cH_0$ to $\cH_1$ is a closed  linear subspace in $\cH_0\oplus\cH_1$. The set of all closed linear relations from $\cH_0$ to $\cH_1$ (in $\cH$) will be denoted by $\C (\cH_0,\cH_1)$ ($\C(\cH)$). A closed linear operator $T$ from $\cH_0$ to $\cH_1$  is
identified  with its graph $\text {gr}\, T\in\CA$.

For a linear relation $T\in\C (\cH_0,\cH_1)$  we denote by $\dom T,\,\ran T,
\,\ker T$ and $\mul T$  the domain, range, kernel and the multivalued part of
$T$ respectively. Recall that $\mul T$ ia a subspace in $\cH_1$ defined by
\begin{gather*}
\mul T:=\{h_1\in \cH_1:\{0,h_1\}\in T\}.
\end{gather*}
Clearly, $T\in \C (\cH_0,\cH_1)$ is an operator if and only if $\mul T=\{0\}$.
The inverse and adjoint linear relations of $T$ are the
relations $T^{-1}\in\C (\cH_1,\cH_0)$ and $T^*\in\C (\cH_1,\cH_0)$ defined by
\begin{gather*}
T^{-1}=\{\{h_1,h_0\}\in\cH_1\oplus\cH_0:\{h_0,h_1\}\in T\}\nonumber\\
T^* = \{\{k_1,k_0\}\in \cH_1\oplus\cH_0:\, (k_0,h_0)-(k_1,h_1)=0, \;
\{h_0,h_1\}\in T\}.
\end{gather*}
Recall also that an operator function $\Phi
(\cd):\bC\setminus\bR\to [\cH]$ is called a Nevanlinna function if it is
holomorphic and satisfies $\im\, \l\cd \im \Phi (\l)\geq 0 $ and $\Phi ^*(\l)=\Phi (\ov \l), \; \l\in\bC\setminus\bR$.
\subsection{Symmetric relations and generalized resolvents}
Recall that a linear relation $A\in\C (\gH)$ is called symmetric (self-adjoint) if $A\subset A^*$ (resp. $A=A^*$). For each symmetric relation $A\in\C (\gH)$ the following decompositions hold
\begin{equation*}
\gH=\gH_0\oplus \mul A, \qquad A={\rm gr}\, A_0\oplus \wh {\mul} A,
\end{equation*}
where $\wh {\mul} A=\{0\}\oplus \mul A$ and $A_0$ is a closed symmetric not necessarily densely defined operator in $\gH_0$ (the operator part of $A$). Moreover, $A=A^*$ if and only if $A_0=A_0^*$.

Let $A=A^*\in \C (\gH)$, let $\cB$ be the Borel $\sigma$-algebra of $\bR$ and let $E_0(\cd):\cB\to [\gH_0]$ be the orthogonal spectral measure of $A_0$. Then the spectral measure $E_A(\cd):\cB\to [\gH]$ of $A$ is defined as $E_A(B)=E_0(B)P_{\gH_0}, \; B\in\cB$.
\begin{definition}\label{def2.0}
Let $\wt A=\wt A^*\in \C(\wt\gH)$ and let $\gH$ be a subspace in $\wt\gH$. The relation $\wt A$ is called $\gH$-minimal if
$\ov{\text{span}} \{\gH,(\wt A-\l)^{-1}\gH: \l\in\CR\}=\wt\gH$.
\end{definition}
\begin{definition}\label{def2.1}
The relations $T_j\in \C (\gH_j), \; j\in\{1,2\},$ are said to be unitarily equivalent (by means of a unitary operator $U\in [\gH_1,\gH_2]$) if $T_2=\wt U
T_1$ with $\wt U=U\oplus U \in [\gH_1^2, \gH_2^2]$.
\end{definition}
Let $A\in\C (\gH)$ be a symmetric relation. Recall the following definitions and results.
\begin{definition}\label{def2.2}
A relation $\wt A=\wt A^*$ in a Hilbert space $\wt \gH \supset \gH$ satisfying $A\subset \wt A$ is called an exit space self-adjoint extension of $A$. Moreover, such an extension $\wt A$ is called minimal if it is $\gH$-minimal.
\end{definition}
In what follows we denote by  $\wt {\rm Self} (A)$ the set of all minimal exit space self-adjoint extensions of $A$. Moreover, we denote by ${\rm Self} (A)$ the set of all extensions $\wt A=\wt A^*\in \C (\gH)$ of $A$ (such an extension is called canonical). As is known, for each $A$ one has $\wt {\rm Self} (A)\neq \emptyset $. Moreover,
${\rm Self} (A)\neq  \emptyset$ if and only if $A$ has equal deficiency indices, in which case ${\rm Self} (A)\subset \wt{\rm Self} (A)$.
\begin{definition}\label{def2.4}
Exit space extensions $\wt A_j=\wt A_j^*\in \C (\wt \gH_j),\; j\in\{1,2\},$ of $A$ are called equivalent (with respect to $\gH$) if there  exists  a unitary operator $V\in
[\wt\gH_1\ominus \gH, \wt\gH_2\ominus \gH]$ such that  $\wt A_1$ and
$\wt A_2$ are unitarily equivalent by means of $U=I_{\gH}\oplus V$.
\end{definition}
\begin{definition}\label{def2.5}
The operator functions $R(\cd):\CR\to [\gH]$ and $F(\cd):\bR\to [\gH]$
are called a generalized resolvent and a spectral function of  $A$
respectively if there exists an exit space extension $\wt A$ of $A$ (in a certain Hilbert space $\wt \gH\supset \gH$)  such that
\begin {gather}
R(\l) =P_\gH (\wt A- \l)^{-1}\up \gH, \quad \l \in \CR\label{2.4}\\
F(t)=P_{\gH}E((-\infty,t))\up\gH, \quad  t\in\bR.\label{2.5}
\end{gather}
Here $P_\gH$ is the orthoprojection in $\wt\gH$ onto $\gH$ and $E(\cd)$ is the spectral measure of $\wt A$.
\end{definition}
In the case $\wt A\in  {\rm Self} (A)$ the equality \eqref{2.4} defines a canonical resolvent $R(\l)=(\wt A-\l)^{-1}$ of A.
\begin{proposition}\label{pr2.6}
Each generalized resolvent $R(\l)$ of $A$ is generated by some (minimal) extension $\wt A\in \wt{\rm Self} (A)$. Moreover,  the extensions $\wt A_1,\, \wt A_2\in \wt {\rm Self} (A)$ inducing the same generalized resolvent $R(\cd)$ are equivalent.
\end{proposition}
In the sequel we suppose   that  a generalized resolvent $R(\cd)$ and a spectral function $F(\cd)$ are generated by an extension $\wt A\in \wt {\rm Self} (A)$. Moreover, we identify equivalent extensions. Then by Proposition \ref{pr2.6} the equality \eqref{2.4} gives a bijective correspondence between generalized resolvents $R(\l)$ and extensions
$\wt A\in \wt {\rm Self} (A)$, so that each $\wt A\in \wt {\rm Self} (A)$  is uniquely  defined by the corresponding generalized resolvent  \eqref{2.4} (spectral function \eqref{2.5}).

It follows from  \eqref{2.4} and \eqref{2.5} that the generalized
resolvent $R(\cd)$ and the spectral function $F(\cd)$ generated by
an  extension $\wt A\in \wt {\rm Self} (A)$  are related  by
\begin {equation*}
R(\l)=\int_{\bR}\frac {d F(t)} {t-\l}, \quad \l\in\bR.
\end{equation*}
Moreover, setting  $\wt\gH_0=\wt\gH\ominus \mul \wt A$ one gets from
\eqref{2.5} that
\begin {equation}\label{2.7}
F(\infty)(:=s-\lim\limits_{t\to +\infty}F(t) )=P_\gH
P_{\wt\gH_0}\up\gH.
\end{equation}
\subsection{The spaces $\cL^2(\Si;\cH)$ and $L^2(\Si;\cH)$ }\label{sub2.3}
Let $\cH$ be a finite dimensional Hilbert space.  A non-decreasing operator function $\Si(\cd): \bR\to [\cH]$ is called a
distribution function if it is left continuous and satisfies $\Si(0)=0$.
\begin{theorem}\label{th2.8} $\,$\cite[ch. 3.15]{DunSch}, \cite{Kac50}
Let $\Si(\cd): \bR\to [\cH]$  be a distribution function. Then:
\begin{enumerate}\def\labelenumi{\rm (\arabic{enumi})}
\item
There exist a scalar measure $\s$ on Borel sets of $\bR$ and a function $\Psi:\bR\to [\cH]$ (uniquely defined by $\s$ up to $\s$-a.e.) such that $\Psi (s)\geq 0$ $\s$-a.e. on $\bR$, $\s([\a,\b))<\infty$ and $\Si(\b)-\Si(\a)=\int\limits_{[\a,\b)}\Psi(s)\, d \s (s)$ for any finite interval $[\a,\b)\subset\bR$.
\item
The set  $\cL^2(\Si;\cH)$ of all Borel-measurable functions $f(\cd):\bR\to \cH$ satisfying
\begin {equation*}
||f||_{\cL^2(\Si;\cH)}^2=\int_\bR (d\Si(s)f(s),f(s)):=\int_\bR(\Psi(s)f(s),f(s))_{\cH}\, d\s (s)<\infty
\end{equation*}
is a semi-Hilbert space with the semi-scalar product
\begin {equation*}
(f,g)_{\cL^2(\Si;\cH)}=\int_\bR (d\Si(s)f(s),g(s)):=\int_\bR(\Psi(s)f(s),g(s))_{\cH}\,d\s (s), \quad f,g\in \cL^2(\Si;\cH).
\end{equation*}
Moreover, different measures $\s$ from statement {\rm (1)} give rise to the same space  $\cL^2(\Si;\cH)$.
\end{enumerate}
\end{theorem}
\begin{definition}\label{def2.9}$\,$\cite{DunSch,Kac50}
The Hilbert space $L^2(\Si;\cH)$ is a Hilbert space of all equivalence classes in $\cL^2(\Si;\cH)$ with respect to the seminorm $||\cd||_{\cL^2(\Si;\cH)}$.
\end{definition}
In the following we denote by $\pi_\Si$ the quotient map from $\cL^2(\Si;\cH)$ onto $L^2(\Si;\cH)$. Moreover, we denote by $\cL^2_{loc}(\Si;\cH)$ the set of all functions $g\in \cL^2(\Si;\cH)$ with the compact support and we put $L^2_{loc}(\Si;\cH):=\pi_\Si \cL^2_{loc}(\Si;\cH)$.

With a distribution function $\Si(\cd)$ one associates the
multiplication operator $\L=\L_\Si$ in $L^2(\Si;\cH)$ defined by
\begin{gather}
\dom \Lambda_\Si=\{\wt f\in L^2(\Si;\cH):s f(s)\in \cL^2(\Si;\cH) \;\;\text{for some (and hence for all)}\;\; f(\cd)\in \wt f\}\nonumber\\
\Lambda_\Si \wt f=\pi_{\Si}(sf(s)), \;\; \wt f\in\dom\Lambda_\Si,\quad f(\cd)\in\wt f.\label{2.10}
\end{gather}
As is known, $\Lambda_\Si^*=\Lambda_\Si$ and the spectral measure $E_\Si$ of $\Lambda_\Si$ is given by
\begin {equation}\label{2.11}
E_\Si(B)\wt f= \pi_\Si (\chi_B(\cd)f(\cd)), \quad B\in\cB,\;\; \wt f \in  L^2(\Si;\cH),\;\; f(\cd)\in \wt f,
\end{equation}
where $\chi_B(\cd)$ is the indicator of the Borel set $B$.

Let $\cK, \;\cK'$ and $\cH$ be finite dimensional Hilbert spaces  and let $\Si(s)(\in [\cH])$ be a distribution function. For Borel measurable functions $Y(s)(\in [\cH,\cK]), \; g(s)(\in \cH)$ and $Z(s)(\in [\cK',\cH]), \; s\in\bR,$ we let
\begin{gather}
\int_\bR Y(s)d\Si(s)g(s):=\int_\bR Y(s)\Psi(s)g(s)\, d\s (s)\, (\in \cK)\label {2.12}\\
\int_\bR Y(s)d\Si(s)Z(s):=\int_\bR Y(s)\Psi(s)Z(s)\, d\s (s)\, (\in [\cK',\cK]),\label {2.13}
\end{gather}
where $\s$ and $\Psi(\cd)$ are defined in Theorem \ref{th2.8}, (1).
\subsection{The classes $\wt R_+(\cH_0,\cH_1)$ and $\wt R(\cH)$}
Let $\cH_0$ be a Hilbert space, let $\cH_1$ be a subspace in $\cH_0$ and let
$\tau =\{\tau_+,\tau_-\}$ be a collection of holomorphic functions
$\tau_\pm(\cd):\bC_\pm\to\CA$. In the paper we systematically deal with
collections $\pair$ of the special class $\wt R_+(\cH_0,\cH_1)$. Definition and
detailed characterization of this class can be found in our paper
\cite{Mog13.2} (see also \cite{Mog06.1,Mog11,AlbMalMog13}, where the notation
$\wt R (\cH_0,\cH_1)$ were used instead of $\wt R_+(\cH_0,\cH_1)$). If
$\dim\cH_1<\infty$, then according to \cite{Mog13.2} the collection $\pair\in
\wt R_+(\cH_0,\cH_1)$ admits the representation
\begin {equation}\label{2.15}
\tau_+(\l)=\{(C_0(\l),C_1(\l));\cH_0\}, \;\;\l\in\bC_+; \;\;\;\;
\tau_-(\l)=\{(D_0(\l),D_1(\l));\cH_1\}, \;\;\l\in\bC_-
\end{equation}
by means of two pairs of holomorphic operator functions
\begin {equation*}
(C_0(\l),C_1(\l)):\cH_0\oplus\cH_1\to\cH_0, \;\;\l\in\bC_+,\;\; \text{and}\;\;
(D_0(\l),D_1(\l)):\cH_0\oplus\cH_1\to\cH_1, \;\;\l\in\bC_-
\end{equation*}
(more precisely, by equivalence classes of such pairs). The equalities
\eqref{2.15} mean that
\begin {equation*}
\begin{array}{c}
\tau_+(\l)=\{\{h_0,h_1\}\in\cH_0\oplus\cH_1: C_0(\l)h_0+C_1(\l)h_1=0\},
\;\;\;\l\in\bC_+\\
\tau_-(\l)=\{\{h_0,h_1\}\in\cH_0\oplus\cH_1: D_0(\l)h_0+D_1(\l)h_1=0\},
\;\;\;\l\in\bC_-.
\end{array}
\end{equation*}
In \cite{Mog13.2} the class $\wt R_+(\cH_0,\cH_1)$ is characterized both in
terms of $\CA$-valued functions $\tau_\pm(\cd)$ and in terms of operator
functions $C_j(\cd)$ and $D_j(\cd), \; j\in\{0,1\},$ from \eqref{2.15}.

If $\cH_1=\cH_0=:\cH$, then the class $\wt R(\cH):=\wt R_+ (\cH,\cH)$ coincides
with the well-known class of Nevanlinna $\C (\cH)$-valued functions $\tau
(\cd)$ (see, for instance, \cite{DM00}). In this case the collection
\eqref{2.15} turns into the Nevanlinna pair
\begin {equation}\label{2.16}
\tau(\l)=\{(C_0(\l),C_1(\l));\cH\}, \quad \l\in\CR,
\end{equation}
with $C_0(\l),C_1(\l)\in [\cH]$. Recall also that the subclass $\wt
R^0(\cH)\subset \wt R(\cH)$ is defined as the set of all $\tau(\cd)\in\wt
R(\cH)$ such that $\tau(\l)\equiv \t(=\t^*), \; \l\in\CR$. This implies that
$\tau(\cd)\in \wt R^0(\cH)$  if and only if
\begin {equation}\label{2.17}
\tau(\l)\equiv \{(C_0,C_1);\cH\},\quad \l\in\CR,
\end{equation}
with some operators $C_0,C_1\in [\cH]$ satisfying  $\im (C_1C_0^*)=0 $ and
$0\in\rho (C_0\pm i C_1)$ (for more details see e.g. \cite[Remark
2.5]{AlbMalMog13}).
\subsection{Boundary triplets and Weyl functions}
Here we recall definitions of a boundary triplet and the corresponding Weyl function of a symmetric relation  following \cite{DM91,Mal92,Mog06.2,Mog13.2}.

Let $A$ be a closed  symmetric linear relation in the Hilbert space $\gH$, let $\gN_\l(A)=\ker (A^*-\l)\; (\l\in\bC)$ be a defect subspace of $A$, let $\wh\gN_\l(A)=\{\{f,\l f\}:\, f\in \gN_\l(A)\}$ and let $n_\pm (A):=\dim \gN_\l(A)\leq\infty, \; \l\in\bC_\pm,$ be deficiency indices of $A$.

Next, assume that $\cH_0$ is a Hilbert space,  $\cH_1$ is a subspace
in $\cH_0$ and   $\cH_2:=\cH_0\ominus\cH_1$, so that $\cH_0=\cH_1\oplus\cH_2$.
Denote by $P_j$ the orthoprojection   in $\cH_0$ onto $\cH_j,\; j\in\{1,2\} $.

\begin{definition}\label{def2.10}
 A collection $\Pi_+=\bta$, where
$\G_j: A^*\to \cH_j, \; j\in\{0,1\},$ are linear mappings, is called a boundary
triplet for $A^*$, if the mapping $\G :\wh f\to \{\G_0 \wh f, \G_1 \wh f\}, \wh
f\in A^*,$ from $A^*$ into $\cH_0\oplus\cH_1$ is surjective and the following
Green's identity holds
\begin {equation*}
(f',g)-(f,g')=(\G_1  \wh f,\G_0 \wh g)_{\cH_0}- (\G_0 \wh f,\G_1 \wh
g)_{\cH_0}+i (P_2\G_0 \wh f,P_2\G_0 \wh g)_{\cH_2}
\end{equation*}
 holds for all $\wh
f=\{f,f'\}, \; \wh g=\{g,g'\}\in A^*$.
\end{definition}
According to \cite{Mog06.2} a boundary triplet $\Pi_+=\bta$ for $A^*$ exists if and only if $n_-(A)\leq n_+(A)$, in which case $\dim \cH_1=n_-(A)$ and $ \dim \cH_0 =n_+(A)$.
\begin{proposition}\label{pr2.11}$\,$ \cite{Mog06.2}
Let  $\Pi_+=\bta$ be a boundary triplet for $A^*$.  Then  the equalities
\begin {gather*}
\G_1\up \wh\gN_\l(A)=M_+(\l)\G_0\up \wh\gN_\l(A),\quad\l\in\bC_+\\
(\G_1+i P_2\G_0)\up \wh\gN_\l(A)= M_-(\l)P_1\G_0\up\wh\gN_\l(A),\quad\l\in\bC_-
\end{gather*}
correctly define the (holomorphic) operator functions $M_{+}(\cdot):\bC_+\to [\cH_0,\cH_1]$ and $M_{-}(\cdot):\bC_-\to [\cH_1,\cH_0]$ satisfying $M_+^*(\ov\l)=M_-(\l), \; \l\in\bC_-$.
\end{proposition}
\begin{definition}\label{def2.12}$\,$\cite{Mog06.2}
The operator functions  $M_\pm(\cd)$ defined in Proposition
\ref{pr2.11}  are called  the Weyl functions corresponding to the boundary triplet $\Pi_+$.
\end{definition}
 \begin{theorem}\label{th2.12.1}$\,$\cite{Mog06.2}
Let $A$ be a closed symmetric linear relation in $\gH$, let
$\Pi_+=\bta$ be a boundary triplet for $A^*$ and let $M_+(\cd)$ be the corresponding Weyl function. If $\pair\in\wt R_+(\cH_0,\cH_1)$ is a collection of holomorphic pairs \eqref{2.15}, then for every
$g\in\gH$ and $\l\in\CR$ the abstract boundary value problem
\begin{gather}
\{f,\l f+g\}\in A^*\label{2.19}\\
C_0(\l)\G_0\{f,\l f+g\}-C_1(\l)\G_1\{f,\l f+g\}=0, \quad
\l\in\bC_+
\label{2.20}\\
D_0(\l)\G_0\{f,\l f+g\}-D_1(\l)\G_1\{f,\l f+g\}=0, \quad
\l\in\bC_-\label{2.21}
\end{gather}
has a unique solution $f=f(g,\l)$ and the equality
$R(\l)g:=f(g,\l)$ defines a generalized resolvent $R(\l)=R_\tau
(\l)$ of $A$. Moreover, $0\in\rho (\tau_+(\l) +M_+(\l))$ and the
following Krein-Naimark formula for resolvents is valid:
\begin {equation}\label{2.22}
R_\tau(\l)=(A_0-\l)^{-1} -\g_+(\l)(\tau_+(\l)
+M_+(\l))^{-1}\g_-^*(\ov\l), \;\;\;\l\in\bC_+
\end{equation}
Conversely, for each generalized resolvent $R(\l)$ of $A$ there
exists a unique $\tau\in \wt R_+(\cH_0,\cH_1)$ such that $R(\l)=R_\tau (\l)$ and, consequently, the equality \eqref{2.22}  is valid.
\end{theorem}
\begin{remark}\label{rem2.12.2}
It follows from Theorem \ref{th2.12.1} that the boundary value
problem \eqref{2.19}--\eqref{2.21} as well as formula for
resolvents \eqref{2.22} give a parametrization of
all generalized resolvents
\begin {equation}\label{2.23}
R(\l)=R_\tau(\l)=P_\gH (\wt A^\tau-\l)^{-1}\up \gH, \quad
\l\in\CR,
\end{equation}
and, consequently, all (minimal) exit space self-adjoint
extensions $\wt A = \wt A^\tau$ of $A$ by means of an abstract
boundary parameter $\tau\in \wt R_+(\cH_0,\cH_1)$.
\end{remark}
\begin{theorem}\label{th2.12.3}
Let under the assumptions of Theorem \ref{th2.12.1} $\pair\in\wt R_+(\cH_0,\cH_1)$  be a collection of holomorphic pairs  \eqref{2.15} and  let $\wt A^\tau\in \C (\wt \gH)$ be the corresponding exit space self-adjoint extension of  $A$ (see remark \ref{rem2.12.2}). Then:

{\rm (1)} The equalities
\begin {gather}
\Phi_\tau(\l):=P_1(C_0(\l)-C_1(\l)M_+(\l))^{-1}C_1
(\l),\quad \l\in\bC_+\label{2.24}\\
\wh\Phi_\tau(\l)=M_+(\l)(C_0(\l)-C_1(\l)M_+(\l))^{-1}C_0(\l)\up\cH_1,\quad \l\in\bC_+\label{2.25}
\end{gather}
define holomorphic $[\cH_1]$-valued functions $\Phi_\tau(\cd)$ and $\wh \Phi_\tau(\cd)$ on $\bC_+$ satisfying $\im \Phi_\tau(\l)\geq 0$ and $\im \wh\Phi_\tau(\l)\geq 0, \; \l\in\bC_+$ . Hence
there exist  strong limits
\begin {gather}
\cB_\tau:=s-\lim_{y\to +\infty} \tfrac 1 {i y}
P_1(C_0(i y)-C_1(i y )M_+(i y))^{-1}C_1(i y)\label{2.26}\\
\wh\cB_\tau:=s-\lim_{y\to +\infty} \tfrac 1 {i y}M_+(i y)(C_0(i y)-C_1(i y )M_+(i y))^{-1}C_0(i y)\up\cH_1\label{2.27}
\end{gather}

\rm{(2)}  The  extension $\wt A^\tau$ satisfies  $\mul
\wt A^\tau =\mul A$ if and only if $\cB_{\tau}=\wh\cB_{\tau}=0$
\end{theorem}
\begin{proof}
Statement (1) for $\Phi_\tau(\l)$ was proved in \cite[Theorem 4.8]{Mog13.2}.

Next assume that
\begin {gather*}
C_0(\l)=(C_{01}(\l),C_{02}(\l)):\cH_1\oplus\cH_2\to\cH_0,\quad D_0(\l)
=(D_{01}(\l),D_{02}(\l)):\cH_1\oplus\cH_2\to\cH_1\\
M_+(\l)=(M(\l),N_+(\l)):\cH_1\oplus\cH_2\to\cH_1,\quad M_-(\l)=
(M(\l),N_-(\l))^\top:\cH_1\to\cH_1\oplus\cH_2
\end{gather*}
are the block-matrix representations of $C_0(\l), \; C_1(\l)$ and $M_\pm(\l)$. Moreover,let
\begin {gather*}
\wh C_0(\l)=( C_1(\l),C_{02}(\l)):\cH_1\oplus\cH_2\to\cH_0; \quad \wh C_1(\l)=-C_{01}(\l), \;\;\l\in\bC_+\\
\wh M_+(\l)=(-M^{-1}(\l),-M^{-1}(\l)N_+(\l) ):\cH_1\oplus\cH_2\to\cH_1,\;\;\l\in\bC_+
\end{gather*}
Then according to \cite{Mog13.2} the equalities
\begin {equation}\label{2.31}
\wh\Phi_\tau(\l):=P_1(\wh C_0(\l)-\wh C_1(\l)\wh M_+(\l))^{-1}\wh C_1
(\l),\;\; \l\in\bC_+; \quad \wh\Phi_\tau(\l):=\wh \Phi_\tau^*(\ov\l),\;\; \l\in\bC_-
\end{equation}
define a Nevanlinna function  $\wh\Phi_\tau(\cd):\CR\to [\cH_1]$ (i.e., a holomorphic function $\wh\Phi_\tau(\cd)$ such that $\im\l\cd\im \wh\Phi_\tau(\l)\geq 0$ and $\wh\Phi_\tau^*(\l)=\wh\Phi_\tau(\ov\l), \; \l\in\CR$). The  immediate checking shows that
\begin {equation*}
(P_2-M_+(\l))^{-1}=-M^{-1}(\l)P_1-M^{-1}(\l)N_+(\l)P_2+P_2
\end{equation*}
and, consequently, $P_1(P_2-M_+(\l))^{-1}=\wh M_+(\l)$ (here $M_+(\l)$ is considered as the operator in $\cH_0$). This and \eqref{2.31} imply  that for each $\l\in\bC_+$
\begin {gather*}
\wh\Phi_\tau(\l)=-P_1\left(C_1(\l)P_1+C_{02}(\l)P_2+C_{01}(\l)P_1(P_2-
M_+(\l))^{-1}\right)^{-1}C_{01}(\l)=\\
-P_1(P_2-M_+(\l))\bigl( (C_1(\l)P_1+C_{02}(\l)P_2)(P_2-
M_+(\l))+ C_{01}(\l)P_1\bigr)^{-1}C_{01}(\l)=\\
M_+(\l) (C_{02}(\l)P_2-C_1(\l)M_+(\l)+C_{01}(\l)P_1)^{-1} C_{01}(\l)=\\
M_+(\l)(C_0(\l)-C_1(\l)M_+(\l))^{-1}C_0(\l)\up \cH_1.
\end{gather*}
Thus the restriction of $\wh \Phi_\tau(\cd)$ on $\bC_+$ admits the representation \eqref{2.25}, which yields statement (1) for $\wh \Phi_\tau(\l)$.

It was shown in \cite{Mog13.2} that the second equality in \eqref{2.31} can be written as
\begin {equation*}
\wh\Phi_\tau(\l):=M(\l)(D_{01}(\l)-D_1(\l)M(\l)-i D_{02}(\l)N_-(\l))^{-1}
D_{01}(\l), \quad \l\in\bC_-
\end{equation*}
Therefore by \eqref{2.27} one has
\begin {gather*}
\wh \cB_\tau=s-\lim\limits_{y\to +\infty} \tfrac 1 {i y} \wh\Phi_\tau(i y)=
s-\lim\limits_{y\to -\infty} \tfrac 1 {i y} \wh\Phi_\tau(i y)=\\
s-\lim\limits_{y\to -\infty} \tfrac 1 {i y} M(iy)(D_{01}(i y)-D_1(i y)M(i y)-i D_{02}(i y)N_-(i y))^{-1}D_{01}(i y).
\end{gather*}
Now statement (2) follows from \cite[Theorem 4.9]{Mog13.2}.
\end{proof}
\begin{remark}\label{rem2.13}
  (1) If $\cH_0=\cH_1:=\cH$, then   the boundary triplet in the sense
of Definition \ref{def2.10} turns into the boundary triplet
$\Pi=\{\cH,\G_0,\G_1\}$ for $A^*$ in the sense of \cite{GorGor,Mal92}. In this case $n_+(A)=n_-(A)(=\dim \cH)$ and  $M_\pm(\cd)$ turn into the Weyl function $M(\cd):\CR\to [\cH]$  introduced in \cite{DM91,Mal92}. Moreover, in this case $M(\cd)$ is a Nevanlinna operator function.

 In the sequel a boundary triplet $\Pi=\bt$ in the
sense of \cite{GorGor,Mal92} will be  called an \emph{ordinary boundary
triplet} for $A^*$.

(2) Let $n_+(A)=n_-(A)$, let  $\Pi=\bt$ be an ordinary boundary
triplet for $A^*$ and let $M(\cd)$ be the corresponding Weyl function. Then an abstract boundary parameter $\tau$ in Theorem \ref{th2.12.1} is a Nevanlinna operator pair $\tau\in \wt R (\cH)$ of the form \eqref{2.16} and the equalities \eqref{2.26} and \eqref{2.27} take the form
\begin {gather}
\cB_\tau=s-\lim_{y\to \infty} \tfrac 1 {i y}
(C_0(i y)-C_1(i y )M(i y))^{-1}C_1(i y)\label{2.33}\\
\wh\cB_\tau=s-\lim_{y\to \infty} \tfrac 1 {i y}M(i y)(C_0(i y)-C_1(i y )M(i y))^{-1}C_0(i y).\label{2.34}
\end{gather}
Note that for this case Theorem \ref{th2.12.3} was proved in \cite{DM00,DM09}.

\end{remark}
\section{First-order symmetric systems }
\subsection{Notations}
Let $\cI=[ a,b\rangle\; (-\infty < a< b\leq\infty)$ be an interval of the real
line (the symbol $\rangle$ means that the endpoint $b<\infty$  might be either
included  to $\cI$ or not). For a given finite-dimensional Hilbert space $\bH$
denote by $\AC$ the set of functions $f(\cd):\cI\to \bH$ which are absolutely
continuous on each segment $[a,\b]\subset \cI$.

Next assume that $\D(\cd)$ is an $[\bH]$-valued Borel measurable  function on
$\cI$ integrable on each compact interval $[a,\b]\subset \cI$ and such that
$\D(t)\geq 0$. Denote  by $\lI$  the semi-Hilbert  space of  Borel measurable
functions $f(\cd): \cI\to \bH$ satisfying $||f||_\D^2:=\int\limits_{\cI}(\D
(t)f(t),f(t))_\bH \,dt<\infty$ (see e.g. \cite[Chapter 13.5]{DunSch}).  The
semi-definite inner product $(\cd,\cd)_\D$ in $\lI$ is defined by $
(f,g)_\D=\int\limits_{\cI}(\D (t)f(t),g(t))_\bH \,dt,\; f,g\in \lI$. Moreover,
let $\LI$ be the Hilbert space of the equivalence classes in $\lI$ with respect
to the semi-norm $||\cd||_\D$ and let $\pi_\D$ be the quotient map from $\lI$ onto
$\LI$.

For a given finite-dimensional Hilbert space $\cK$ we denote by $\lo{\cK}$ the set
of all Borel measurable  operator-functions $F(\cd): \cI\to [\cK,\bH]$ such
that $F(t)h\in \lI, \; h\in\cK$.
\subsection{Symmetric systems}
In this subsection we provide some known results on symmetric systems of
differential equations following  \cite{GK, Kac03, LesMal03, Orc}.

Let $H$ and $\wh H$ be  finite-dimensional Hilbert spaces and let
\begin {equation}\label{3.0.1}
H_0=H\oplus\wh H, \quad \bH=H_0\oplus  H=H\oplus\wh H \oplus H.
\end{equation}
Let as above $\cI=[ a,b\rangle\; (-\infty < a< b\leq\infty)$ be an interval in
$\bR$ . Moreover, let $B(\cd)$ and $\D(\cd)$ be $[\bH]$-valued Borel measurable functions on $\cI$ integrable on each compact interval $[a,\b]\subset \cI$ and
satisfying $B(t)=B^*(t)$ and $\D(t)\geq 0$ a.e. on $\cI$ and let $J\in [\bH]$ be  operator \eqref{1.3}.

A first-order symmetric  system on an interval $\cI$ (with the regular endpoint
$a$) is a system of differential equations of the form
\begin {equation}\label{3.1}
J y'-B(t)y=\D(t) f(t), \quad t\in\cI,
\end{equation}
where $f(\cd)\in \lI$. Together   with \eqref{3.1} we consider also the
homogeneous system
\begin {equation}\label{3.2}
J y'(t)-B(t)y(t)=\l \D(t) y(t), \quad t\in\cI, \quad \l\in\bC.
\end{equation}
A function $y\in\AC$ is a solution of \eqref{3.1} (resp. \eqref{3.2}) if
equality \eqref{3.1} (resp. \eqref{3.2} holds a.e. on $\cI$. A function
$Y(\cd,\l):\cI\to [\cK,\bH]$ is an operator solution of  equation \eqref{3.2}
if $y(t)=Y(t,\l)h$ is a (vector) solution of this equation for every $h\in\cK$
(here $\cK$ is a Hilbert space with $\dim\cK<\infty$).

The following lemma will be useful in the sequel.
\begin{lemma}\label{lem3.0}
Let $\cK$ be a finite dimensional Hilbert space, let $Y(\cd,\cd):\cI\times \bR\to [\cK,\bH]$ be an operator function such that $Y(\cd,s)$ is a solution of \eqref{3.2} and $Y(a,\cd)$ is a continuous function on $\bR$ and let $\Si(\cd):\bR\to [\cK]$ be a distribution function. Then for each function $g\in\cL_{loc}^2(\Si;\cK)$ the equality
\begin {equation}\label{3.3}
f(t)=\int_\bR Y(t,s)\,d\Si(s)g(s), \quad t\in\cI
\end{equation}
defines a function $f(\cd)\in \AC$ such that
\begin {equation}\label{3.4}
f'(t)=-J\int_\bR (B(t)+s\D(t))Y(t,s)\, d\Si(s) g(s)\quad ({\rm a.e. on}\;\; \cI ).
\end{equation}
\end{lemma}
\begin{proof}
According to \eqref{2.12} the equality \eqref{3.3} means
\begin {equation}\label{3.5}
f(t)=\int_\bR Y(t,s) \Psi(s)g(s)\, d\s (s), \quad t\in\cI,
\end{equation}
where $\Psi$ and $\s$ are defined in Theorem \ref{th2.8}, (1). Since $Y(t,s)$ satisfies
\begin {equation}\label{3.6}
Y(t,s)=Y(a,s)-J\int_{[a,t)}(B(u)+s \D(u))Y(u,s)du, \quad t\in\cI,
\end{equation}
it follows that $Y(\cd,\cd)$ is a continuous function on $\cI\times \bR$. Moreover, one can easily prove that $\int\limits_\bR ||\Psi(s)g(s)||\, d\s (s)<\infty$. Therefore the integral in \eqref{3.5} exists and
\begin {equation}\label{3.7}
\int\limits_{[a,t)\times \bR}||(B(u)+s \D(u))Y(u,s)\Psi(s)g(s)||\,du\,d\s(s)<\infty.
\end{equation}
It follows from \eqref{3.7} and the Fubini  theorem that
\begin{gather}
\int_\bR\left(\int_{[a,t)}(B(u)+s \D(u))Y(u,s)\Psi(s)g(s)\,du\right)d\s(s)=\label{3.8}\\
\int_{[a,t)}\left(\int_\bR(B(u)+s \D(u))Y(u,s)\Psi(s)g(s)\,d\s(s)\right)du.\nonumber
\end{gather}
Now combining \eqref{3.5} with \eqref{3.6} and taking \eqref{3.8} into account one gets
\begin {equation*}
f(t)=C-J \int_{[a,t)}\left(\int_\bR (B(u)+s \D(u))Y(u,s) \Psi(s)g(s)\,d\s(s)\right)\, du,
\end{equation*}
where $C=\int\limits_{\bR} Y(a,s)\Psi(s)g(s)\,d\s(s)$. Hence $f(\cd)\in\AC$ and \eqref{3.4} holds.
\end{proof}
In what follows  we always assume  that  system \eqref{3.1} is definite in the sense of the following definition.
 \begin{definition}\label{def3.1}$\,$\cite{GK}
Symmetric system \eqref{3.1} is called definite if for each $\l\in\bC$ and each
solution $y$ of \eqref{3.2} the equality $\D(t)y(t)=0$ (a.e. on $\cI$) implies
$y(t)=0, \; t\in\cI$.
\end{definition}
Introduce also the following definition.
\begin{definition}\label{def3.1a}
System \eqref{3.1} will be called absolutely definite if
\begin {equation*}
\mu_1(\{t\in\cI: \; \text{the operator}\;\; \D(t)\;\;\text{is invertible}\})>0,
\end{equation*}
where $\mu_1$ is the Lebesgue  measure on $\cI$.
\end{definition}
Clearly, each absolutely definite system is definite. Moreover, one can easily construct definite, but not absolutely definite system  \eqref{3.1} (even with $B(t)\equiv 0$ and continuous $\D(t)$).

As it is known \cite{Orc, Kac03, LesMal03} definite system \eqref{3.1} gives
rise to the \emph{maximal linear relations} $\tma$ and $\Tma$  in  $\lI$ and
$\LI$, respectively. They are given by
\begin {equation*}
\begin{array}{c}
\tma=\{\{y,f\}\in(\lI)^2 :y\in\AC \;\;\text{and}\;\; \qquad\qquad\qquad\qquad \\
\qquad\qquad\qquad\qquad\qquad\quad  J y'(t)-B(t)y(t)=\D(t) f(t)\;\;\text{a.e.
on}\;\; \cI \}
\end{array}
\end{equation*}
and $\Tma=\{\{\pi_\D y,\pi_\D f\}:\{y,f\}\in\tma\}$. Moreover the Lagrange's identity
\begin {equation*}
(f,z)_\D-(y,g)_\D=[y,z]_b - (J y(a),z(a)),\quad \{y,f\}, \; \{z,g\} \in\tma.
\end{equation*}
holds with
\begin {equation}\label{3.9}
[y,z]_b:=\lim_{t \uparrow b}(J y(t),z(t)), \quad y,z \in\dom\tma.
\end{equation}
Formula \eqref{3.9} defines the skew-Hermitian  bilinear form $[\cd,\cd]_b $ on
$\dom \tma$. By using this form one defines the \emph{minimal relations} $\tmi$
in $\lI$ and $\Tmi$ in $\LI$ via
\begin {equation*}
\tmi=\{\{y,f\}\in\tma: y(a)=0 \;\; \text{and}\;\; [y,z]_b=0\;\;\text{for
each}\;\; z\in \dom \tma \}.
\end{equation*}
and $\Tmi=\{\{\pi_\D y,\pi_\D f\}:\{y,f\}\in\tmi\} $. According to \cite{Orc,Kac03,
LesMal03} $\Tmi$ is a closed symmetric linear relation in $\LI$ and
$\Tmi^*=\Tma$.

Denote by $\cN_\l,\; \l\in\bC,$ the linear space of solutions of the
homogeneous system \eqref{3.2} belonging to $\lI$ and let $\gN_\l(\Tmi)$ be the defect subspace of $\Tmi$. Since system \eqref{3.1} is definite, it follows that $\dim \gN_\l(\Tmi)=\dim \cN_\l$. Hence $\Tmi$ has finite (not necessarily equal) deficiency indices $n_\pm(\Tmi)=\dim \cN_\l\leq\dim \bH,\; \l\in\bC_\pm$.

The following assertion is immediate from definitions of $\Tmi$ and $\Tma$.
\begin{assertion}\label{ass3.2}
{\rm (1)} The multivalued part $\mul \Tmi$ of the minimal relation $\Tmi$ is the set of all $\wt f\in\LI$  such that for some (and hence for all) $f\in\wt f$ the  solution $y$ of Eq. \eqref{3.2} with $y(a)=0$ satisfies
\begin {equation*}
y\in\lI, \quad \D(t)y(t)=0\;\; (\text{a.e. on }\;\;\cI)\;\;\text{and}\;\; [y,z]_b=0, \;\;z\in\dom\tma.
\end{equation*}
{\rm (2)} The equality $\mul \Tmi=\mul\Tma$ holds if and only if for each function $y\in\dom\tma$ the equality $\D(t)y(t)=0$ (a.e. on $\cI$) yields $y(a)=0$ and $[y,z]_b=0, \;z\in\dom\tma$.
\end{assertion}
\begin{remark}\label{rem3.3}
It  is known (see e.g. \cite{LesMal03}) that the maximal relation $\Tma$
induced by the definite symmetric system \eqref{3.1} possesses the following
 property: for any $\{\wt y, \wt f \}\in \Tma $ there exists
a unique function $y\in \AC \cap \lI $ such that $y\in \wt y$ and $\{y,f\}\in
\tma$ for any $f\in\wt f$. Below we associate such a function $ y\in \AC \cap
\lI$ with each pair $\{\wt y, \wt f\}\in\Tma$.
\end{remark}
\subsection{Decomposing boundary triplets}\label{sub3.3}
In this subsection we provide some results from \cite{AlbMalMog13}.
\begin{lemma} \label{lem3.4}
If $n_-(\Tmi)\leq n_+(\Tmi)$, then there exist a finite dimensional  Hilbert  space $\wt \cH_b$, a subspace $\cH_b\subset \wt \cH_b$  and a surjective linear mapping
\begin{gather}\label{3.12}
\G_b=\begin{pmatrix}\G_{0b}\cr  \wh\G_b \cr  \G_{1b}\end{pmatrix} :\dom\tma\to
\wt \cH_b\oplus\wh H \oplus \cH_b
\end{gather}
such that for all $y,z \in \dom\tma$ the following identity is valid
\begin {equation} \label{3.13}
 [y,z]_b=(\G_{0b}y,\G_{1b}z)_{\wt\cH_b}-(\G_{1b}y,\G_{0b}z)_{\wt\cH_b}
+i(P_{\cH_b^\perp}\G_{0b}y, P_{\cH_b^\perp}\G_{0b}z)_{\wt\cH_b}+ i (\wh\G_b y,
\wh\G_b z)_{\wh H}
\end{equation}
(here $\cH_b^\perp=\wt\cH_b\ominus \cH_b$). Moreover, in the case $n_+(\Tmi)= n_-(\Tmi)$ (and only in this case) one has $\wt \cH_b=\cH_b$ and the identity \eqref{3.13} takes the form
\begin {equation*}
[y,z]_b=(\G_{0b}y,\G_{1b}z)_{\cH_b}-(\G_{1b}y,\G_{0b}z)_{\cH_b}+  i (\wh\G_b y,
\wh\G_b z)_{\wh H}
\end{equation*}
\end{lemma}
Up to the end of this subsection we assume that:

(A1) $n_-(\Tmi)\leq n_+(\Tmi)$

(A2) $\wt\cH_b$ and $\cH_b(\subset\wt\cH_b)$ are finite dimensional  Hilbert  spaces and $\G_b$ is a surjective linear mapping \eqref{3.12} satisfying \eqref{3.13}.

For a function $y\in\dom\tma$ we let
\begin{gather}
\G_0'y=\{ - y_1(a),\;i(\wh y(a)-\wh\G_b y),\;
\G_{0b}y\}\in  H\oplus\wh H\oplus\wt\cH_b\label{3.13.1}\\
\G_1'y=\{y_0(a),\; \tfrac 1 2(\wh y(a)+\wh\G_b y), \;
-\G_{1b}y\}\in H\oplus\wh H\oplus \cH_b\label{3.13.2},
\end{gather}
where $y_0(a), \wh y(a)$ and $y_1(a)$ are taken from the representation
$y(t)=\{y_0(t),\,\wh y(t), \, y_1(t) \}(\in H\oplus\wh H \oplus
H)$ of $y$.
\begin{proposition}\label{pr3.5}
A collection $\Pi_+=\bta$ with
\begin{gather}
\cH_0=\underbrace{H\oplus\wh H}_{H_0}\oplus\wt\cH_b=H_0 \oplus\wt\cH_b, \qquad
\cH_1=\underbrace{H\oplus\wh H}_{H_0}\oplus\cH_b=H_0\oplus \cH_b\label{3.14}\\
\G_0\{\wt y, \wt f\}=\G_0'y,\qquad \G_1\{\wt y, \wt f\}=\G_1'y,\quad \{\wt y, \wt
f\}\in\Tma \label{3.15}
\end{gather}
is a (so called decomposing) boundary triplet for $\Tma$.
In  \eqref{3.15} $y\in\dom\tma$ is a function corresponding to $\{\wt y, \wt f\}\in\Tma$ in accordance with Remark \ref{rem3.3}.
\end{proposition}
\begin{proposition}\label{pr3.6}
Let $\cP_0,\; \wh\cP$ and $\cP_1$ be the orthoprojectors in $\bH$ onto the first, second and third component respectively in the decomposition $\bH=H\oplus\wh H\oplus H$ (see \eqref{3.0.1}). Then:

{\rm (1)} For every $\l\in\bC_+$  there exists a unique pair of operator solutions $v_0(\cd,\l)\in\lo{H_0}$ and $u(\cd,\l)\in\lo{\wt\cH_b}$ of Eq. \eqref{3.2} satisfying
\begin{gather*}
\cP_1 v_0(a,\l)=-P_{H_0,H}, \quad i (\wh \cP v_0(a,\l)h_0-\wh\G_b
(v_0(\cd,\l)h_0))=P_{H_0,\wh H}, \quad \G_{0b} (v_0(\cd,\l)h_0)=0\\
\cP_1 u(a,\l)=0, \quad i(\wh\cP u(a,\l)h_b-\wh\G_b (u(\cd,\l)h_b))=0,\quad \G_{0b} (u(\cd,\l)h_b)=h_b
\end{gather*}
for all $h_0\in H_0$ and $h_b\in \wt\cH_b$. Here $P_{H_0,H}\in [H_0,H]$ and $P_{H_0,\wh H}\in [H_0,\wh H]$ are the orthoprojectors in $H_0(=H\oplus\wh H)$ onto $H$ and $\wh H$ respectively.

{\rm (2)} The Weyl function $M_+(\cd)$ of the decomposing boundary triplet $\Pi_+=\bta$  for $\Tma $ (see Definition \ref{def2.12}) admits the representation
\begin{gather}
M_+(\l)=\begin{pmatrix} m_0(\l )& M_{2+}(\l) \cr M_{3+}(\l) & M_{4+}(\l)
\end{pmatrix}: \underbrace{H_0\oplus\wt\cH_b}_{\cH_0}\to \underbrace{H_0
\oplus\cH_b}_{\cH_1},\;\;\;\l\in\bC_+\label{3.18}\\
m_0(\l)=(\cP_0 +\wh \cP)v_0(a,\l)+\tfrac i 2 P_{\wh H}, \;\;\;\;\; M_{2+}(\l)
=(\cP_0+\wh\cP)u(a,\l)\label{3.19}\\
M_{3+}(\l)=-\G_{1b}v_0(\l), \qquad M_{4+}(\l)=-\G_{1b}u(\l).\label{3.20}
\end{gather}
If in addition $n_+(\Tmi)=n_-(\Tmi)$, then: {\rm (1)} the solutions $v_0(\cd,\l)$ and $u(\cd,\l)$ are defined for  $\l\in\CR$;  {\rm (2)} $\Pi_+$ turns into an ordinary boundary triplet $\Pi=\bt$ for $\Tma$ with $\cH=H_0\oplus\cH_b$ and the  Weyl function $M(\cd)$ of the triplet $\Pi$  is of the form
\begin {gather}
M(\l)=\begin{pmatrix} m_0(\l )& M_{2}(\l) \cr M_{3}(\l) & M_{4}(\l)
\end{pmatrix}: H_0\oplus \cH_b\to H_0\oplus\cH_b,
\;\;\;\l\in\CR\label{3.21}\\
m_0(\l)=(\cP_0+\wh\cP)v_0(a,\l)+\tfrac i 2 P_{\wh H},\qquad
M_{2}(\l)=(\cP_0+\wh\cP)u(a,\l),\label{3.22}\\
M_{3}(\l)=-\G_{1b}v_0(\l), \qquad M_{4}(\l)=-\G_{1b}u(\l).\label{3.23}
\end{gather}
\end{proposition}
\begin{remark}\label{rem3.7}
According to \cite[Proposition 4.5]{AlbMalMog13} the Weyl function $M_-(\cd)$ of the decomposing boundary triplet $\Pi_+$ also admits the representation in terms of boundary values of respective operator solutions of Eq. \eqref{3.2} (cf. \eqref{3.18}--\eqref{3.20} for $M_+(\l)$).
\end{remark}
\section{Pseudospectral and spectral functions of symmetric systems}
\subsection{$q$-pseudospectral functions}
In what follows we put $\gH:=\LI$ and denote by $\gH_b$ the set of all $\wt f\in\gH$ with the following property: there exists $\b_{\wt f}\in\cI$ such that for some (and hence for all) function $f\in\wt f$ the equality $\D(t)f(t)=0$ holds a.e. on $(\b_{\wt f}, b)$.

Denote by $Y_0(\cd,\l)$ the $[\bH]$-valued operator solution of \eqref {3.2} satisfying $Y_0(a,\l)=I_\bH$. With each $\wt f\in\gH_b$ we associate the function $\wh f(\cd):\bR\to\bH$ given by
\begin{equation}\label{4.1}
\wh f(s)=\int_\cI Y_0^*(t,s)\D(t)f(t)\,dt,\quad f(\cd)\in\wt f.
\end{equation}
By using the well known properties of the solution $Y_0(\cd,\l)$ one can easily prove that $\wh f(\cd)$ is a continuous (and even holomorphic) function on $\bR$.

Recall that an operator $V\in [\gH_1,\gH_2]$ is called a partial isometry if $||Vf||=||f||$ for all $f\in\gH_1\ominus\ker V$.
\begin{definition}\label{def4.1}
A distribution function $\Si(\cd):\bR\to [\bH]$ will be called a $q$-pseudospectral function of the system \eqref {3.1} if $\wh f\in\lS$ for all $\wt f\in\gH_b$ and the operator $V_b \wt f:=\pi_\Si\wh f, \; \wt f\in\gH_b,$ admits a continuation to a partial isometry $V=V_\Si\in [\gH,\LS]$.

The operator $V=V_\Si$ will be  called the Fourier transform corresponding to $\Si(\cd)$.
\end{definition}
Clearly, if $\Si(\cd)$ is a $q$-pseudospectral function, then for each $f(\cd)\in\lI$ there exists a unique $\wt g (=V_\Si \pi_\D f)\in \LS$ such that for each function $g(\cd)\in\wt g$ one has
\begin{equation*}
\lim\limits_{\b\uparrow b}\Bigl |\Bigl |g(\cd)-\int\limits_{[a,\b)} Y_0^*(t,\cd)\D(t)f(t)\,dt\Bigr|\Bigr|_{\lS}=0.
\end{equation*}
\begin{proposition}\label{pr4.3}
Let $\Si(\cd)$ be a $q$-pseudospectral function and let $V=V_{\Si}$ be the corresponding Fourier transform. Then for each $\wt g\in L_{loc}^2(\Si;\bH)$ the function
\begin{equation*}
f_{\wt g}(t):=\int_\bR Y_0(t,s)\, d \Si(s) g(s), \quad g(\cd)\in\wt g
\end{equation*}
belongs to $\lI$ and $V^*\wt g=\pi_\D f_{\wt g}(\cd)$. Therefore
\begin{equation*}
V^*\wt g=\pi_\D \left( \int_\bR Y_0(\cd,s)\, d \Si(s) g(s)\right),\quad \wt g\in \LS, \;\; g(\cd)\in\wt g,
\end{equation*}
where the integral converges in the seminorm of $\lI$.
\end{proposition}
\begin{proof}
According to Lemma \ref{lem3.0} $f_{\wt g}(\cd)$ is a continuous $\bH$-valued function on $\cI$ and by \eqref{2.12}
\begin{equation}\label{4.2}
f_{\wt g}(t)=\int_\bR Y_0(t,s)\Psi(s)g(s)\, d \s (s), \quad g(\cd)\in\wt g,
\end{equation}
where $\s$  and $\Psi$ are defined in Theorem \ref{th2.8}, (1).

Let $f_*(\cd)\in\lI$ be a function such that $\pi_\D f_*(\cd)=V^*\wt g$. Moreover, let $h\in\bH$, let $\d\subset \cI$ be a compact interval and let $f(t)=\chi_\d(t)h(\in\lI)$. We show that
\begin{equation}\label{4.3}
\int_{\cI}(f(t),\D(t)f_{\wt g}(t))_\bH\,dt= \int_{\cI}(f(t),\D(t) f_*(t))_\bH\,dt.
\end{equation}
In view of \eqref{4.2} one has
\begin{equation}\label{4.4}
\int_{\cI}(f(t),\D(t)f_{\wt g}(t))_\bH\,dt=\int_\cI \left (\int_\bR (\D(t)f(t),Y_0(t,s)\Psi(s)g(s))_\bH\, d\s (s) \right)dt.
\end{equation}
Since $Y_0(\cd,\cd)$ is a continuous function on $\cI\times\bR$, it follows that
\begin{equation*}
\int\limits_{\cI\times\bR}\left|(\D(t)f(t), Y_0(t,s)\Psi(s)g(s))_\bH\right|\, dt d \s(s)<\infty.
\end{equation*}
Therefore by the Fubini theorem one has
\begin{gather*}
\int_\cI \left (\int_\bR (\D(t)f(t),Y_0(t,s)\Psi(s)g(s))_\bH\, d\s (s) \right)dt=\\
\int_{\bR}\left( \int_\cI (\D(t)f(t),Y_0(t,s)\Psi(s)g(s))_\bH\,dt \right)\,d\s(s)=\\
\int_{\bR}\left( \int_\cI (\Psi(s)Y_0^*(t,s)\D(t)f(t),g(s))_\bH\,dt \right)\,d\s(s)=\\
\int_{\bR}\left(\Psi(s) \int_\cI Y_0^*(t,s)\D(t)f(t)\,dt,g(s)\right)_\bH \,d\s(s)=(V\pi_\D f,\wt g)_{\LS}=\\(\pi_\D f,V^*\wt g)_{\gH}=
\int_{\cI}(f(t),\D(t)f_*(t))dt.
\end{gather*}
 Combining these relations with \eqref{4.4} one gets the equality \eqref{4.3}.

It follows from \eqref{4.3} that $\D(t) f_{\wt g}(t)=\D(t)f_*(t)$ (a.e. on $\cI$). Hence $f_{\wt g}(\cd)\in\lI$ and $\pi_\D f_{\wt g}(\cd)=\pi_\D f_*(\cd)=V^*\wt g$.
\end{proof}
\begin{proposition}\label{pr4.3a}
Let $\Si(\cd)$ be a $q$-pseudospectral function of the system \eqref{3.1} and let $L_0$ be a subspace in $\LS$ given by $L_0=V_\Si\gH$. Then the multiplication operator $\L_\Si$ is $L_0$-minimal (in the sense of Definition \ref{def2.0}).
\end{proposition}
We omit  the proof of Proposition \ref{pr4.3a}, because it is similar to that of \cite[Proposition 6.9]{AlbMalMog13}.

Let $V_\Si$ be the Fourier transform corresponding to the $q$-pseudospectral function $\Si(\cd)$ and let $\gH_0=\gH\ominus\ker V_\Si, \; L_0=V_\Si\gH(=V_\Si\gH_0)$ and $L_0^\perp =\LS\ominus L_0$. Then
\begin{equation}\label{4.5}
\gH=\ker V_\Si\oplus\gH_0, \qquad \LS=L_0\oplus L_0^\perp.
\end{equation}
Assume also that
\begin{equation}\label{4.5a}
\wt\gH_0:=\gH_0\oplus L_0^\perp, \quad \wt\gH:=\overbrace{\ker V_\Si\oplus\gH_0}^{\gH}\oplus L_0^\perp=\gH\oplus L_0^\perp=\ker V_\Si\oplus\wt\gH_0
\end{equation}
and let $\wt V=\wt V_\Si\in [\wt\gH_0,\LS]$ be a unitary operator of the form
\begin{equation}\label{4.6}
\wt V=(V_{0,\Si},\, I_{L_0^\perp} )
:\gH_0\oplus L_0^\perp\to \LS,
\end{equation}
where $V_{0,\Si}=V_\Si\up\gH_0$ and $ I_{L_0^\perp} $   is an embedding operator from $L_0^\perp$ to $\LS$. Since $\gH\subset \wt\gH$, one may  consider $\Tmi$ as a linear relation in $\wt \gH$.
\begin{lemma}\label{lem4.4}
Let $\Si(\cd)$ be a $q$-pseudospectral function of the system \eqref{3.1} and let $\wt V$ be a unitary operator \eqref{4.6}. Moreover, let $(\Tmi)^*_{\wt\gH}\in \C (\wt\gH)$ be a linear relation adjoint to $\Tmi$ in $\wt\gH$ and let $\L=\L_\Si$ be the multiplication operator in $\LS$. Then the equalities
\begin{equation}\label{4.7}
\wt f =\wt V^* \wt g, \qquad \wt T_0 \wt f=\wt V^*\L \wt g, \quad \wt g\in\dom \L
\end{equation}
define a self-adjoint operator $\wt T_0$ in $\wt\gH_0$ such that $\wt T_0\subset (\Tmi)^*_{\wt\gH}$.
\end{lemma}
\begin{proof}
It is easily seen that $(\Tmi)^*_{\wt\gH}=\Tma\oplus (L_0^\perp)^2$. Moreover, in view of \eqref{4.6} and the equality $V_{0,\Si}^*\wt g=V_\Si^* \wt g, \; \wt g\in\LS,$ one has
\begin{equation*}
\wt V^* \wt g =V_\Si^* \wt g +P_{L_0^\perp}\wt g, \quad\wt g\in\LS.
\end{equation*}
Therefore \eqref{4.7} can be written as
\begin{equation*}
\wt f = V_\Si^* \wt g+P_{L_0^\perp}\wt g, \qquad  \wt T_0 \wt f= V_\Si^*\L \wt g+P_{L_0^\perp}\L \wt g,\quad \wt g\in\dom \L.
\end{equation*}
Thus to prove the inclusion $\wt T_0\subset (\Tmi)^*_{\wt\gH}$ it is sufficient to show that $\{V_\Si^* \wt g, V_\Si^*\L \wt g\}\in \Tma$ for all $\wt g\in\dom \L$.

Let $\wt g\in \dom \L, \; g(\cd)\in\wt g$ and let $E(\cd)=E_\Si(\cd)$ be the spectral measure of $\L$. Then by \eqref{2.11} and \eqref{2.10} for each compact interval $\d\subset \bR$ one has $E(\d)\wt g=\pi_\Si (\chi_\d (\cd)g(\cd))$ and $\L E(\d)\wt g=\pi_\Si(s\chi_\d(s)g(s))$. Therefore according to Proposition \ref{pr4.3} $V_\Si^* E(\d)\wt g=\pi_\D y(\cd)$ and $V_\Si^* \L E(\d)\wt g=\pi_\D f(\cd)$, where
\begin{equation*}
y(t)=\int_\bR Y_0(t,s)d\Si(s) \chi_\d(s)g(s), \quad f(t)=\int_\bR s Y_0(t,s)d\Si(s) \chi_\d(s)g(s).
\end{equation*}
It follows from Lemma \ref{lem3.0} that $y\in \AC$ and
\begin {equation*}
y'(t)=-J\int_\bR (B(t)+s\D(t))Y_0(t,s)\, d\Si(s) \chi_\d(s) g(s)\quad ({\rm a.e. on}\;\; \cI ).
\end{equation*}
Therefore
\begin {equation*}
J y'(t)-B(t)y(t)=\D(t)\int_\bR s Y_0(t,s)\, d\Si(s) \chi_\d(s) g(s)=\D(t)f(t) \quad ({\rm a.e. on}\;\; \cI )
\end{equation*}
and, consequently, $\{y,f\}\in\tma$. Hence $\{V_\Si^* E(\d)\wt g, V_\Si^* \L E(\d)\wt g\}(=\{\pi_\D y(\cd),\pi_\D f(\cd)\})\in\Tma$ and passage to the limit when $\d\to \bR$ yields the required inclusion $\{V_\Si^* \wt g, V_\Si^*\L \wt g\}\in \Tma$.
\end{proof}
\begin{proposition}\label{pr4.5}
For each $q$-pseudospectral function $\Si(\cd)$ of the system \eqref{3.1} the corresponding Fourier transform $V_\Si$ satisfies
\begin {equation}\label{4.8}
\mul\Tmi\subset \ker V_\Si
\end{equation}
(for $\mul\Tmi$ see Assertion \ref{ass3.2}, (1)).
\end{proposition}
\begin{proof}
Let  $\wt T_0=\wt T_0^*$ be the operator in $\wt\gH_0$ defined in Lemma \ref{lem4.4} and let $(\wt T_0)^*_{\wt\gH}$ be the linear relation adjoint to $\wt T_0$ in $\wt\gH$ . Then $(\wt T_0)^*_{\wt\gH}=\wt T_0\oplus (\ker V_\Si)^2$ and the inclusion $\wt T_0\subset (\Tmi)^*_{\wt\gH}$ yields
\begin {equation}\label{4.9}
\Tmi\subset \wt T_0\oplus (\ker V_\Si)^2.
\end{equation}
Let $n\in\mul\Tmi$. Then $\{0,n\}\in\Tmi$ and by \eqref{4.9} $\{0,n\}\in \wt T_0\oplus (\ker V_\Si)^2$. Therefore there exist $f\in\dom \wt T_0$ and $g,g'\in \ker V_\Si$ such that
\begin {equation*}
f+g=0, \qquad \wt T_0 f+g'=n.
\end{equation*}
Since $f\in \wt\gH_0, \; g\in\ker V_\Si$ and $\wt\gH_0\perp \ker V_\Si$ (see \eqref{4.5a}), it follows that $f=g=0$. Therefore $\wt T_0 f=0$ and hence $n=g'\in \ker V_\Si$. This yields the inclusion \eqref{4.8}.
\end{proof}
\begin{remark}\label{rem4.5a}
According to \cite[Lemma 5]{LanTex84} the equality
\begin {equation}\label{4.9a}
\Phi_s f=\int_\cI Y_0^*(t,s)\D(t)f(t)\, dt, \quad \wt f\in\dom\Tmi \cap\gH_b, \;\; s\in\bR
\end{equation}
defines a directing mapping $\Phi$ of $\Tmi$ in the sense of \cite{LanTex84}.  By using this fact and Theorem 1 from \cite{LanTex84} one can prove the inclusion \eqref{4.8} for $q$-pseudospectral functions $\Si(\cd)$ satisfying the additional condition $||V_\Si \wt f||=||\wt f||, \; \wt f\in\dom\Tmi$.
\end{remark}
\begin{definition}\label{def4.6}
A $q$-pseudospectral function $\Si(\cd)$ of the system \eqref{3.1} will be called a pseudospectral function if the corresponding Fourier transform $V_\Si$ satisfies $\ker V_\Si=\mul \Tmi$.
\end{definition}
\begin{definition}\label{def4.7}
A distribution function $\Si(\cd):\bR\to [\bH]$ will be called a spectral function of the system \eqref{3.1} if $\wh f\in \lS$ and the Parseval equality $||\wh f||_{\lS}=||\wt f||_\gH$ holds for all $\wt f\in\gH_b$ (here $\wh f$ is the Fourier transform \eqref{4.1}).
\end{definition}
It follows from Proposition \ref{pr4.5} that a pseudospectral function is a $q$-pseudospectral function $\Si(\cd)$ with the minimally possible $\ker V_\Si$. Moreover, the same proposition yields the following assertion.
\begin{assertion}\label{ass4.8}
A distribution function $\Si(\cd):\bR\to [\bH]$ is a spectral function of the system \eqref{3.1} if and only if it is a pseudospectral function with $\ker V_{\Si}(=\mul\Tmi)=\{0\}$ (that is, with the isometry $V_\Si$).
\end{assertion}
In the following we put $\gH_0:=\gH\ominus \mul\Tmi$, so that
\begin {equation}\label{4.10.1}
\gH=\mul\Tmi\oplus \gH_0.
\end{equation}
Moreover, for a pseudospectral function $\Si(\cd)$ we denote by $V_0=V_{0,\Si}$ the isometry from $\gH_0$ to $\LS$ given by
\begin {equation}\label{4.10.2}
V_{0,\Si}:=V_\Si\up \gH_0.
\end{equation}
Clearly, $V_\Si$ admits the representation
\begin {equation}\label{4.10.2a}
V_\Si=(0, \,V_{0,\Si}):\mul\Tmi\oplus \gH_0\to \LS
\end{equation}
\subsection{Pseudospectral functions and extensions of the minimal relation}
For a Hilbert space $\wt\gH\supset\gH$ we put $\wt\gH_0:=\wt\gH\ominus \mul\Tmi$, so that
\begin {equation}\label{4.10.3}
\wt\gH=\mul\Tmi\oplus \wt\gH_0.
\end{equation}
It is clear that $\gH_0\subset\wt\gH_0$ (for $\gH_0$ see \eqref{4.10.1}).
\begin{definition}\label{def4.9}
A minimal exit space extension $\wt T=\wt T^*\in\C (\wt \gH)$ of $\Tmi$ is referred to the class $\Sel$ if $\mul \wt T=\mul\Tmi$. Moreover, we denote by $\Selo$ the set of all canonical extensions $\wt T=\wt T^*$ of $\Tmi$ satisfying  $\mul \wt T=\mul\Tmi$.
\end{definition}
Clearly, $\Sel\subset \wt {\rm Self}(\Tmi)$ and $\Selo \subset {\rm Self}(\Tmi)$.  Moreover, if $\mul\Tmi=\{0\}$, then $\Sel$ ($\Selo$) is the set of all extensions $\wt T\in \wt {\rm Self}(\Tmi)$ (resp. $\wt T\in  {\rm Self}(\Tmi)$) which are the operators.

For each $\wt T\in \Sel$ we will denote by $\wt T_0$ the operator part of $\wt T$, so that $\wt T_0$ is a self-adjoint operator in $\wt\gH_0$. Let $E_0(\cd)$ be the orthogonal spectral measure of $\wt T_0$ and let $F_0(\cd):\bR\to [\gH_0]$ be a distribution function given by
\begin {equation}\label{4.10.4}
F_0(t)=\wt P_{\gH_0}E_0((-\infty,t))\up\gH_0, \quad t\in\bR,
\end{equation}
where $\wt P_{\gH_0}$ is the orthoprojector in $\wt\gH_0$ onto $\gH_0$. It is clear that a spectral function $F(\cd)$ of $\Tmi$ generated by $\wt T$ is of the form
\begin {equation}\label{4.10.5}
F(t)= {\rm diag}\, (F_0(t), \, 0):\gH_0\oplus\mul \Tmi\to \gH_0\oplus\mul \Tmi.
\end{equation}
Next assume that $\wt T\in \Sel$ and that $F(\cd)$  is a spectral function of $\Tmi$ generated by $\wt T$. Moreover, let $\Si(\cd)$ be a pseudospectral function of the system \eqref{3.1}.
\begin{definition}\label{def4.10}
We write $\wt T=\wt T_\Si$ if
\begin {equation} \label{4.10.6}
((F(\b)-F(\a))\wt f,\wt f)_\gH=\int_{[\a,\b)} (d\Si(s)\wh f(s),\wh f(s)), \;\;\;\;\wt f\in \gH_b, \;\;\;\; -\infty<\a<\b<\infty.
\end{equation}
\end{definition}
\begin{proposition}\label{pr4.12}
For each pseudospectral function $\Si(\cd)$ of the system \eqref{3.1} there exists a unique (up to the equivalence) exit space extension $\wt T\in\Sel$ such that $\wt T=\wt T_\Si$. Moreover,   there exists a unitary operator $\wt V \in [\wt \gH_0,\LS]$ such that $\wt V\up \gH_0=V_{0,\Si}$ and the operators $\wt T_0$ and $\L_\Si$  are unitarily equivalent by means of $\wt V$.
\end{proposition}
\begin{proof}
For a given pseudospectral function $\Si(\cd)$ we put $L_0=V_\Si \gH_0$ and $L_0^\perp =\LS\ominus L_0$, so that $\LS=L_0\oplus L_0^\perp$. Assume also that
\begin{equation}\label{4.10.11}
\wt\gH_0:=\gH_0\oplus L_0^\perp, \quad \wt\gH:=\mul\Tmi\oplus\gH_0\oplus L_0^\perp=\mul\Tmi\oplus\wt\gH_0
\end{equation}
and let $\wt V\in [\wt\gH_0,\LS]$ be a unitary operator  \eqref{4.6}. Since $\ker V_\Si=\mul\Tmi$, it follows from Lemma \ref{lem4.4} that the equalities \eqref{4.7} define a self-adjoint operator $\wt T_0$ in $\wt\gH_0$. Moreover, in view of \eqref{4.7} the operators  $\wt T_0$ and $\L=\L_\Si$ are unitarily equivalent by means of $\wt V$. Hence the spectral measure $E_0(\cd)$ of $\wt T_0$ satisfies
\begin{equation}\label{4.10.13}
E_0([\a,\b))=\wt V^* E_\Si ([\a,\b))\wt V, \quad -\infty<\a<\b<\infty. \end{equation}
Observe also that $\wt V\gH_0=V_\Si \gH_0=L_0$ and by Proposition \ref{pr4.3a} the operator $\L_\Si$ is $L_0$-minimal. Therefore the operator $\wt T_0$ is $\gH_0$-minimal.

It follows from the second equality in \eqref{4.10.11} that
\begin{equation*}
\wt T:=(\{0\}\oplus \mul\Tmi)\oplus \wt T_0
\end{equation*}
is a self-adjoint linear relation in $\wt \gH$ with the operator part $\wt T_0$ and $\mul\wt T=\mul\Tmi$. Moreover, $\{0\}\oplus \mul\Tmi\subset \Tmi\subset (\Tmi)^*_{\wt\gH}$ and by Lemma \ref{lem4.4} $\wt T_0\subset (\Tmi)^*_{\wt\gH}$. Hence $\wt T\subset (\Tmi)^*_{\wt\gH}$ and, consequently, $\Tmi\subset \wt T$. Observe also that the relation $\wt T$ is $\gH$-minimal, since the operator $\wt T_0$ is $\gH_0$-minimal. Hence $\wt T\in \Sel$.

Next we show that $\wt T=\wt T_\Si$. Let $F(\cd)$ be a spectral function of $\Tmi$ generated by $\wt T$ and let $F_0(\cd)$ be given by \eqref{4.10.4}. By using \eqref{4.10.13} and \eqref{4.6} one can easily show that
\begin{equation*}
F_0(\b)-F_0(\a)=\wt P_{\gH_0}E_0([\a,\b))\up\gH_0=V_{0,\Si}^* E_\Si([\a,\b))V_{0,\Si}, \quad -\infty<\a<\b<\infty.
\end{equation*}
Therefore by \eqref{4.10.5} and \eqref{4.10.2a} one has
\begin{equation*}
F(\b)-F(\a)=V_{\Si}^*E_\Si([\a,\b))V_{\Si}, \quad -\infty<\a<\b<\infty,
\end{equation*}
which is equivalent to \eqref{4.10.6}. Hence  $\wt T=\wt T_\Si$.

Finally, uniqueness of $\wt T=\wt T_\Si$ directly follows from \eqref{4.10.6} and $\gH$-minimality of $\wt T$.
\end{proof}
The following corollary is immediate from Proposition \ref{pr4.12}.
\begin{corollary}\label{cor4.13}
Let $\Si(\cd)$ be a pseudospectral function of the system \eqref{3.1}. Then $V_{0,\Si}$ is a unitary operator from $\gH_0$ onto $\LS$ if and only if $n_+(\Tmi)=n_-(\Tmi)$ and $\wt T_\Si\in \Selo$. If these conditions are satisfied, then the operators $\wt T_0$ and $\L_\Si$ are unitarily equivalent by means of $V_{0,\Si}$.
\end{corollary}
\begin{remark}\label{rem4.13a}
Applying \cite[Theorem 1]{LanTex84} to the directing mapping \eqref{4.9a} one can give another proof of Proposition \ref{pr4.12}.
\end{remark}

The following theorem is well known (see e.g. \cite{Bru78,DLS93,Sht57}).
\begin{theorem}\label{th4.14}
Let $Y_0(\cd,\l)$ be the $[\bH]$-valued operator solution of Eq. \eqref{3.2}
satisfying $Y_0(a,\l)=I_{\bH}$. Then for each generalized resolvent $R(\l)$ of
$\Tmi$ there exists a unique operator function $\Om (\cd):\CR\to [\bH]$ such
that for each  $ \wt f\in\LI$ and $\l\in\CR$
\begin {equation}\label{4.15}
R(\l)\wt f=\pi_\D\left(\int_\cI Y_0(\cd,\l)(\Om(\l)+\tfrac 1 2 \, {\rm
sgn}(t-x)J)Y_0^*(t,\ov\l)\D(t) f(t)\,dt \right), \quad f\in\wt f.
\end{equation}
Moreover, $\Om(\cd)$ is a Nevanlinna operator function.
\end{theorem}
\begin{definition}\label{def4.15}$\,$\cite{Bru78,Sht57}
The operator function $\Om(\cd) $ is called the characteristic matrix of the
symmetric system \eqref{3.1} corresponding  to the generalized resolvent $R(\l)$.
\end{definition}
Since  $\Om(\cd)$ is a Nevanlinna function, it follows that the equality (the Stieltjes formula)
\begin {equation}\label{4.16}
\Si_\Om(s)=\lim\limits_{\d\to+0}\lim\limits_{\varepsilon\to +0} \frac 1 \pi
\int_{-\d}^{s-\d}\im \,\Om(\s+i\varepsilon)\, d\s.
\end{equation}
defines a distribution $[\bH]$-valued function $\Si_\Om(\cd)$. This function is called a spectral function of $\Om(\cd)$.
\begin{theorem}\label{th4.16}
Let $\wt T\in\Sel$, let $R(\cd)$ be the generalized resolvent of $\Tmi$ generated by $\wt T$, let $\Om(\cd)$ be the characteristic matrix corresponding to $R(\cd)$ and let $\Si_\Om(\cd)$ be the spectral function of $\Om(\cd)$. Then $\Si(\cd)=\Si_\Om(\cd)$ is a pseudospectral function of the system \eqref{3.1} and $\wt T=\wt T_\Si$ (in the sense of Definition \ref{def4.10}). If in addition system \eqref{3.1} is absolutely definite, then $\Si(\cd)=\Si_\Om(\cd)$ is a unique pseudospectral function of this system  satisfying $\wt T=\wt T_\Si$.
\end{theorem}
\begin{proof}
(1) Assume that $\wt T$ is a linear relations in the Hilbert space $\wt \gH\supset \gH$. Let $F(\cd)$ be the spectral function of $\Tmi$ induced by $\wt T$. By using \eqref{4.15} and the Stieltjes-Liv\u{s}ic inversion formula one proves the equality \eqref{4.10.6} for $\Si(\cd)=\Si_\Om(\cd)$ in the same way as Theorem 4 in \cite{Sht57}.

Next assume that  $\gH$ and $\wt\gH$ are decomposed as in \eqref{4.10.1} and \eqref{4.10.3} respectively. It follows from \eqref{4.10.6} and \eqref{2.7} that for any $\wt f\in\gH_b$ one has $\wh f\in\lS$ and $||\wh f||_{\lS}=||P_{\wt\gH_0}\wt f||_{\wt\gH}\leq ||\wt f||_\gH$. Hence the operator $V_b\wt f:=\pi_\Si\wh f,\; \wt f\in\gH_b,$ admits a continuation to an operator $V\in [\gH,\LS]$ satisfying
\begin {equation}\label{4.17}
||V\wt f||_{\LS}= ||P_{\wt\gH_0}\wt f||_{\wt\gH}, \quad \wt f\in\gH.
\end{equation}
It follows from \eqref{4.17}, \eqref{4.10.3} and the inclusion $\gH_0\subset\wt\gH_0$ that $V\wt f=0,\; \wt f\in\mul\Tmi,$ and $||V\wt f||_{\LS}=||\wt f||_{\wt\gH}=||\wt f||_{\gH_0}, \; \wt f\in\gH_0$. Thus $V$ is a partial isometry with $\ker V=\mul\Tmi$ and, consequently, $\Si(\cd)=\Si_\Om(\cd)$ is a pseudospectral function of the system \eqref{3.1}. Moreover, $F(\cd)$ satisfies  \eqref{4.10.6}, so that $\wt T=\wt T_\Si$.

(2) Now assume that system \eqref{3.1} is absolutely definite and show that in this case each  pseudospectral function $\Si(\cd)$ satisfying $\wt T=\wt T_\Si$ coincides with $\Si_\Om(\cd)$. So, let a pseudospectral function $\Si(\cd)$ of the system \eqref{3.1} satisfies \eqref{4.10.6}, let $V_\Si$ be the corresponding Fourier transform and let $E_\Si$ be spectral measure \eqref{2.11}. Then by \eqref{4.10.6} for each finite interval $\d=[\a,\b)\subset\bR$ one has
\begin {equation}\label{4.20}
F(\b)-F(\a)=V_\Si^* E_\Si(\d)V_\Si
\end{equation}
and Proposition \ref{pr4.3} yields
\begin {equation}\label{4.21}
(F(\b)-F(\a))\wt f=\pi_\D \left ( \int_\d Y_0(\cd,s)d\Si(s) \wh f(s)\right ), \quad  \d=[\a,\b)\subset\bR, \;\; \wt f\in\gH_b.
\end{equation}
Substituting \eqref{4.1} into \eqref{4.21} and then using the Fubini theorem one can easily show that
\begin {equation}\label{4.22}
(F(\b)-F(\a))\wt f=\pi_\D \left (\int_\cI K_{\d,\Si}(\cd,u) \D (u) f(u)\, du \right ), \quad \d=[\a,\b)\subset\bR, \;\; \wt f\in\gH_b, \;\; f\in\wt f,
\end{equation}
where
\begin {equation}\label{4.23}
K_{\d,\Si}(t,u)=\int_\d Y_0(t,s)d \Si(s) Y_0^*(u,s), \quad t,u\in\cI.
\end{equation}
Let $K_{\d,\Si_\Om}(t,u)$ be given by \eqref{4.23} with $\Si(s)=\Si_\Om (s)$ and let $K_\d(t,u)=K_{\d,\Si}(t,u)-K_{\d,\Si_\Om}(t,u), \;t,u\in\cI$. It follows from Theorem \ref{th2.8} that there exist a scalar measure $\s$ on $\cB$ and functions $\Psi,\Psi_\Om:\bR\to [\bH]$ such that
\begin {equation}\label{4.23a}
\Si(\b)-\Si(\a)=\int_\d \Psi(s)\, d\s(s)\;\;\text{and}\;\; \Si_\Om(\b)-\Si_\Om(\a)=\int_\d \Psi_\Om(s)\, d\s(s)
\end{equation}
for any finite  $\d=[\a,\b)$. Let $\wt\Psi (s)=\Psi(s)-\Psi_\Om(s)$. Then in view of \eqref{4.23} one has
\begin {equation}\label{4.24}
K_\d(t,u)=\int_\d Y_0(t,s) \wt \Psi(s) Y_0^*(u,s)\, d\s(s), \quad t,u\in\cI,\;\; \d=[\a,\b)\subset\bR.
\end{equation}
Since $\Si_\Om(\cd)$ also satisfies \eqref{4.10.6}, the equality \eqref{4.22} holds with $K_{\d,\Si_\Om}$ in place of $K_{\d,\Si}$. Hence
\begin {equation}\label{4.25}
\pi_\D\left (\int_\cI K_{\d}(\cd,u) \D (u) f(u)\, du \right )=0, \quad \d=[\a,\b)\subset\bR, \;\; f\in\lI, \;\; \pi_\D f\in\gH_b.
\end{equation}
Denote by $F$ ($F'$) the set of all finite intervals $\d=[\a,\b)\subset \bR$ (resp. $\d'=[\a',\b')\subset \cI$) with rational endpoints. Moreover, let $\{e_j\}_1^n$ be a basis in $\bH$. It follows from \eqref{4.25} that for any $\d\in F, \; \d'\in F'$ and $e_j$ there exists a Borel set $B=B(\d,\d', e_j)\subset \cI$ such that $\mu_1(\cI\setminus B)=0$ and
\begin {equation}\label{4.26}
\int_{\d'} \D(t) K_\d(t,u)\D(u)e_j\, du=0, \quad t\in B.
\end{equation}
For each $\d\in F$ put
\begin {equation}\label{4.27}
\wt K_\d(t,u)=\D(t) K_\d(t,u)\D(u)=\int_\d \D(t)Y_0(t,s)\wt\Psi (s) Y_0^*(u,s) \D(u)\, d\s(s)
\end{equation}
and let $B_\d=\{\{t,u\}\in \cI\times\cI: \wt K_\d(t,u)=0\}, \; B_0=\bigcap\limits_{\d\in F}B_\d$. It follows from \eqref{4.26} that $\mu_2(\cI\times\cI\setminus B_\d)=0,\; \d\in F,$ and hence $\mu_2(\cI\times\cI\setminus B_0)=0$ (here $\mu_2$ is the Lebesgue measure on $\cI\times\cI$). Let $X_\D=\{t\in\cI: \D(t)\;\; \text{is invertible}\}$. Since system \eqref{3.1} is absolutely definite, it follows that $\mu_1 (X_\D)>0$. Hence $\mu_2(X_\D\times X_\D)>0$ and, consequently, $(X_\D\times X_\D)\cap B_0\neq \emptyset$. Therefore there exist $t_0$ and $u_0$ in $I$ such that the operators $\D(t_0)$ and $\D(u_0)$ are invertible and the equality
\begin{equation*}
\wt K_\d(t_0,u_0)=\int_\d \D(t_0)Y_0(t_0,s)\wt\Psi (s) Y_0^*(u_0,s) \D(u_0)\, d\s(s)=0
\end{equation*}
holds for all $\d\in F$. Hence $\D(t_0)Y_0(t_0,s)\wt\Psi (s) Y_0^*(u_0,s) \D(u_0)=0 $ ($\s$-a.e. on $\bR$) and invertibility of $Y_0(t_0,s)$ and $Y_0^*(u_0,s)$ yields $\wt\Psi (s) =0$ ($\s$-a.e. on $\bR$). Thus $\Psi(s)=\Psi_\Om(s)$ and by \eqref{4.23a} $\Si(s)=\Si_\om(s)$.
 \end{proof}
The above results show that in the case of the absolutely definite system \eqref{3.1} the equality $\wt T=\wt T_\Si$ gives a bijection between all pseudospectral functions $\Si(\cd)$ and all exit space extensions $\wt T\in\Sel$. The inverse bijection $\Si=\Si_{\wt T}$ is characterized by the following theorem, which is implied immediately by Proposition  \ref{pr4.12}, Theorem \ref{th4.16} and Corollary \ref{cor4.13}.
\begin{theorem}\label{th4.17}
Let system \eqref{3.1} be absolutely definite. Then the equalities \eqref{4.15} and \eqref{4.16} give a bijective correspondence $\Si(\cd)=\Si_{\wt T}(\cd)$ between all extensions $\wt T\in\Sel$ and all pseudospectral functions $\Si(\cd)$. More precisely, let $\wt T\in\Sel$, let $R(\cd)=R_{\wt T}(\cd)$ be the  generalized resolvent of $\Tmi$  induced by $\wt T$, let $\Om(\cd)=\Om_{\wt T}(\cd)$ be the characteristic matrix corresponding to $R_{\wt T}(\cd)$ and let $\Si_{\wt T}(\cd)$ be the spectral function of $\Om_{\wt T}(\cd)$. Then  $\Si_{\wt T}(\cd)$
is a pseudospectral function of the system \eqref{3.1}. Conversely, for each pseudospectral  function $\Si(\cd)$ of the system \eqref{3.1} there exists a unique (up to equivalence) $\wt T\in\Sel$ such that $\Si(\cd)=\Si_{\wt T}(\cd)$.

Moreover, $V_{0,\Si}$ is a unitary operator from $\gH_0$ onto $\LS$ if and only if $n_+(\Tmi)=n_-(\Tmi)$ and $\Si(\cd)=\Si_{\wt T}$ with $\wt T\in\Selo$.
\end{theorem}
Next, combining the results of this subsection with Assertion \ref{ass4.8} one gets the following theorem.
\begin{theorem}\label{th4.18}
The set of spectral functions of the system \eqref{3.1} is not empty if and only if $\mul\Tmi=\{0\}$. If this condition is satisfied, then the set of spectral functions coincides with the set of pseudospectral functions and hence  Proposition \ref{pr4.12},  Theorems \ref{th4.16}, \ref{th4.17} and Corollary \ref{cor4.13} hold with the following replacements: the  phrase  ''spectral function(s)'' instead of ''pseudospectral function(s)''; the Hilbert space $\gH$, the operator $\wt T$ and the isometry $V_\Si$ in place of $\gH_0,\; \wt T_0$ and $V_{0,\Si}$ respectively.
\end{theorem}
\begin{remark}\label{rem4.19}
For a not necessarily absolutely definite system Theorem \ref{th4.17} and the last statement of Theorem \ref{th4.16} could be easily obtained from Theorem 1 in \cite{LanTex84} applied to the directing mapping \eqref{4.9a}. For this purpose it would be needed one of the statements of the mentioned Theorem 1, which is not proved in  \cite{LanTex84} (namely, uniqueness of a spectral function $V$ of $\langle S;\Phi \rangle$ for a given extension $\wt S=\wt S^*$ of $S$, where the notations are taken from \cite{LanTex84}). In fact, we do not know whether Theorem \ref{th4.17} and the last statement of Theorem \ref{th4.16} are valid for not absolutely definite systems \eqref{3.1}.
\end{remark}
\section{Parametrization of pseudospectral and spectral functions}
In the following we suppose that the assumptions (A1) and (A2) from Subsection \ref{sub3.3} are satisfied.
\begin{definition}\label{def5.1}
Let $\cH_0$ and $\cH_1$ be finite dimensional Hilbert spaces \eqref{3.14}. Then a boundary parameter $\tau$ is a collection $\pair\in\wt R_+(\cH_0,\cH_1)$ of the form \eqref{2.15}.
\end{definition}
In the case of equal deficiency indices $n_+(\Tmi)=n_-(\Tmi)$ one has
\begin{equation}\label{5.0}
\wt\cH_b=\cH_b, \qquad \cH_0=\cH_1=:\cH=H_0\oplus\cH_b
\end{equation}
 and a boundary parameter is an operator pair
$\tau\in\wt R(\cH)$ defined by \eqref{2.16}. If in addition $\tau\in\wt R^0(\cH)$, then a boundary parameter will be called self-adjoint. Such a boundary parameter $\tau$ admits the representation as a self-adjoint operator pair \eqref{2.17}.

Let $\pair$ be a boundary parameter \eqref{2.15}. For a given function $f\in\lI$ consider the boundary problem
\begin {gather}
J y'-B(t)y = \l \D(t) y+\D(t)f(t), \quad t\in\cI\label{5.1}\\
C_0(\l)\G_0'y-C_1(\l)\G_1'y=0, \;\;\l\in\bC_+;\qquad D_0(\l)\G_0'y-D_1(\l)\G_1'y=0, \;\;\l\in\bC_-,\label{5.2}
\end{gather}
where $\G_0'y\in\cH_0$ and $\G_1'y\in\cH_1$ are defined by \eqref{3.13.1} and \eqref{3.13.2}. A function $y(\cd,\cd):\cI\tm (\CR)\to\bH$ is called a solution of this problem if for each $\l\in\CR$ the function $y(\cd,\l)$ belongs to $\AC\cap\lI$ and satisfies the equation \eqref{5.1} a.e. on $\cI$ (so that $y\in\dom\tma$) and the boundary conditions \eqref{5.2}.

If $n_+(\Tmi)=n_-(\Tmi)$ and $\tau$ is a boundary parameter
\eqref{2.16}, then \eqref{5.2} takes the form
\begin {gather}\label{5.3}
C_0(\l)\G_0'y-C_1(\l)\G_1' y = 0, \quad \l\in\CR.
\end{gather}
If in addition $\tau$ is a self-adjoint boundary parameter \eqref{2.17}, then \eqref{5.3} turns into  a self-adjoint boundary condition
\begin {gather}\label{5.4}
C_0\G_0' y-C_1 \G_1' y = 0.
\end{gather}
Observe also that in our paper \cite{Mog14} the boundary conditions \eqref{5.2}--\eqref{5.4} were represented in a more compact form.
\begin{theorem}\label{th5.2}
Let $\pair$ be a boundary parameter \eqref{2.15}. Then for every $f\in\lI$ the boundary problem \eqref{5.1},  \eqref{5.2} has a unique solution $y(t,\l)=y_f(t,\l) $ and the equality
\begin {equation*}
R(\l)\wt f = \pi_\D(y_f(\cd,\l)), \quad \wt f\in \LI, \quad f\in\wt f, \quad \l\in\CR
\end{equation*}
defines a generalized resolvent $R(\l)=:R_\tau(\l)$ of $\Tmi$. Conversely, for each generalized resolvent $R(\l)$ of $\Tmi$ there exists a unique boundary parameter $\tau$ such that $R(\l)=R_\tau(\l)$.

If in addition $n_+(\Tmi)=n_-(\Tmi)$, then   the above statements hold with the boundary parameter $\tau$ of the form \eqref{2.16} and
the boundary condition \eqref{5.3} in place of \eqref{5.2}. Moreover,
$R_\tau(\l)$ is a canonical resolvent of $\Tmi$ if and only if $\tau$ is a self-adjoint boundary parameter \eqref{2.17}. In this case
$R_\tau(\l)=(\wt T^\tau - \l)^{-1}$, where
\begin {equation}\label{5.5}
\wt T^\tau=\{\{\wt y,\wt f\}\in\Tma: C_0\G_0'y-C_1\G_1'y=0\}.
\end{equation}
\end{theorem}
\begin{proof}
Let $\Pi_+=\bta$ be the decomposing boundary triplet \eqref{3.14}, \eqref{3.15} for $\Tma$. It follows from \eqref{3.15} and \eqref{3.13.1}, \eqref{3.13.2} that the boundary problem \eqref{5.1}, \eqref{5.2} is equivalent to \eqref{2.19}--\eqref{2.21}.  Now applying  Theorem \ref{th2.12.1}  we arrive at the required statements.
\end{proof}
According to Theorem \ref{th5.2} the boundary problem \eqref{5.1}, \eqref{5.2} induces a bijective correspondence $R(\l)=R_\tau(\l)$ between boundary parameters $\tau$ and generalized resolvents $R(\l)$ of $\Tmi$. In the following we denote by $\wt T^\tau (\in \wt{\rm Self}(\Tmi))$ the extension of $\Tmi$ generating $R_\tau(\cd)$ and by $\Om_\tau(\cd)$ the characteristic matrix corresponding to $R_\tau(\cd)$. Clearly, the equalities $\wt T=\wt T^\tau $ and $\Om(\cd)=\Om_\tau(\cd)(=\Om_{\wt T^\tau }(\cd))$ gives a parametrization of all extensions $\wt T\in \wt{\rm Self}(\Tmi)$ and all characteristic matrices $\Om(\cd)$ of the system \eqref{3.1} respectively by means of a boundary parameter $\tau$.

Let $\cH_0$ and $\cH_1$ be given by \eqref{3.14} and let $\cH_2:=\cH_0\ominus\cH_1(=\wt\cH_b\ominus\cH_b)$, so that $\cH_0=\cH_1\oplus\cH_2$. Denote by $P_j$ the orthoprojector in $\cH_0$ onto $\cH_j, \; j\in\{1,2\}$.

Next, assume that $M_+(\cd):\bC_+\to [\cH_0,\cH_1]$  is the operator function \eqref{3.18}--\eqref{3.20}  (this means that $M_+(\cd)$ is the Weyl functions of the decomposing boundary triplet for $\Tma$) and let  $\pair$ be a boundary parameter \eqref{2.15}.  It follows from Theorem \ref{th2.12.3} that there exist the limits $\cB_\tau$ and $\wh\cB_\tau$ of the form \eqref{2.26} and \eqref{2.27}.
\begin{definition}\label{def5.3}
A boundary parameter $\tau $ will be called admissible if $\cB_\tau=\wh\cB_\tau=0$.
\end{definition}
The following assertions are immediate from the results of \cite{Mog13.2}:

(i) If $\lim\limits_{y\to \infty} \tfrac 1 {i y} M_+(iy)\up \cH_1=0$,
then $\tau$ is admissible if and only if $\cB_\tau=0$.

(ii) Every boundary parameter is admissible if and only if $\mul\Tmi=\mul\Tma$ (see Assertion \ref{ass3.2}, (2)) or  equivalently, if and  only if  $\lim\limits_{y\to \infty} \tfrac 1 {i y} M_+(iy)\up \cH_1=0 $ and
\begin {equation}\label{5.13}
\lim_{y\to +\infty}y \left (\im (M_+(i y)h_0,h_0)_{\cH_0}+\tfrac 1
2 ||P_2 h_0 ||^2\right )=+\infty, \quad h_0\in \cH_0, \quad
h_0\neq 0.
\end{equation}
In the following theorem we describe all pseudospectral functions of the system \eqref{3.1} in terms of the boundary parameter $\tau$.
\begin{theorem}\label{th5.4}
Let system \eqref{3.1} be absolutely definite and let $n_-(\Tmi)\leq n_+(\Tmi)$. Assume also that   $M_+(\cd)$ is the operator function \eqref{3.18}--\eqref{3.20} and let
\begin {gather}
\Om_0(\l)=\begin{pmatrix} m_0(\l) & -\tfrac 1 2  I_{H,H_0}\cr -\tfrac 1 2
P_{H_0,H} & 0\end{pmatrix}:\underbrace {H_0\oplus H}_{\bH}\to \underbrace{
H_0\oplus H}_{\bH}, \quad \l\in\CR\label{5.14}\\
S_1(\l)=\begin{pmatrix} m_0(\l)-\tfrac i 2 P_{\wh H} & M_{2+}(\l) \cr
-P_{H_0,H} & 0 \end{pmatrix}:\underbrace{H_0\oplus
\wt\cH_b}_{\cH_0}\to\underbrace{ H_0\oplus H}_{\bH},\quad \l\in\bC_+ \nonumber\\
S_2(\l)=\begin{pmatrix} m_0(\l)+\tfrac i 2 P_{\wh H} & -I_{H,H_0}\cr M_{3+}(\l) & 0\end{pmatrix}:\underbrace{ H_0\oplus H}_{\bH}\to\underbrace{H_0\oplus
\cH_b}_{\cH_1}, \quad \l\in\bC_+ ,\nonumber
\end{gather}
where $P_{H_0,H}\in [H_0,H]$ is the orthoprojector in $H_0$ onto $H$, $I_{H,H_0}\in [H,H_0]$ is the embedding operator of $H$ into $H_0$ and $P_{\wh H}\in [H_0]$ is the orthoprojector in $H_0$ onto $\wh H$ (see \eqref{3.0.1}). Then the  equalities
\begin {gather}
\Om_\tau(\l)=\Om_0(\l)+S_1(\l)(C_0 (\l)- C_1 (\l)M_+(\l))^{-1}
C_1(\l)S_2(\l), \quad \l\in\bC_+\label{5.17}\\
\Si(s)=\Si_\tau(s)=\lim\limits_{\d\to+0}\lim\limits_{\varepsilon\to +0} \frac 1 \pi
\int_{-\d}^{s-\d}\im \,\Om_\tau(\s+i\varepsilon)\, d\s.\label{5.18}
\end{gather}
establish a bijective correspondence between all admissible boundary parameters $\pair$ defined by \eqref{2.15} and all pseudospectral functions $\Si(\cd)=\Si_\tau(\cd)$ of the system \eqref{3.1}.
\end{theorem}
\begin{proof}
As it was mentioned in the proof of Theorem \ref{th5.2}, the problem \eqref{5.1}, \eqref{5.2} can be represented in terms of the decomposing boundary triplet $\Pi_+$ for $\Tma$ as \eqref{2.19}-\eqref{2.21}. Hence the parametrization of generalized resolvents $R(\l)=R_\tau(\l)$ and the corresponding extensions $\wt T=\wt T^\tau\in\wt{\rm Self}(\Tmi)$ coincides with the parametrization of the same objects in terms of the triplet $\Pi_+$ given in Theorem \ref{th2.12.1}. Therefore according to Theorem \ref{th2.12.3} $\wt T^\tau\in\Sel$ if and only if $\tau$ is admissible. Moreover, the parametrization of all characteristic matrices $\Om_\tau(\cd)=\Om _{\wt T^\tau}(\cd)$ in the form \eqref{5.17} was obtained in \cite[Theorem 4.6]{Mog14}. Combining these facts with Theorem \ref{th4.17} we arrive at the required statement.
\end{proof}
Assume now that $\Tmi $ has equal deficiency indices $n_+(\Tmi)=n_-(\Tmi)$. Then \eqref{5.0} holds and the equalities \eqref{3.21}--\eqref{3.23} define a (Nevanlinna) operator function $M(\cd)$ (the Weyl function of the ordinary decomposing boundary triplet $\Pi=\bt$ for $\Tma$). Observe also that in this case:

(1) a boundary parameter $\tau\in\wt R(\cH)$ is defined by \eqref{2.16} and the equalities \eqref{2.26} and \eqref{2.27} take a simpler form \eqref{2.33} and \eqref{2.34}.

(2) the condition \eqref{5.13} turns into
\begin {gather*}
\lim_{y\to \infty}y \cd\im (M(i y)h,h)=+\infty, \quad h\in \cH, \quad
h\neq 0.
\end{gather*}
\begin{theorem}\label{th5.5}
Let system \eqref{3.1} be absolutely definite and let $n_+(\Tmi)=n_-(\Tmi)$. Moreover, let $M(\cd)$ be given by \eqref{3.21}--\eqref{3.23}, let $\Om_0(\cd)$ be defined by \eqref{5.14} and let
\begin {equation*}
S(\l)=\begin{pmatrix} m_0(\l)-\tfrac i 2 P_{\wh H} & M_{2}(\l) \cr
-P_{H_0,H} & 0 \end{pmatrix}:\underbrace{H_0\oplus
\cH_b}_{\cH}\to\underbrace{ H_0\oplus H}_{\bH},\quad \l\in\CR.
\end{equation*}
Then the equality
\begin {equation}\label{5.20}
\Om_\tau(\l)=\Om_0(\l)+S(\l)(C_0 (\l)- C_1(\l) M(\l))^{-1}
C_1(\l)S^*(\ov\l), \quad \l\in\CR
\end{equation}
together with the Stieljes formula \eqref{5.18} establishes a bijective correspondence between all admissible boundary parameters $\tau$ of the form \eqref{2.16} and all pseudospectral functions $\Si(\cd)=\Si_\tau(\cd)$ of the system.

Moreover, $V_{0,\Si}(\in [\gH_0,\LS])$ is a unitary operator  if and only if $\tau$ is a self-adjoint (admissible) boundary parameter. If this condition is satisfied, then equality \eqref{5.5} defines an extension $\wt T^\tau\in\Selo$ and the operators $\wt T_0^\tau$ (the operator part of $\wt T^\tau$) and $\L_\Si $ are unitarily equivalent by means of $V_{0,\Si}$.
\end{theorem}
\begin{proof}
According to \cite[Theorem 4.9]{Mog14} in the case $n_+(\Tmi)=n_-(\Tmi)$  equality \eqref{5.17} admits the representation \eqref{5.20}. This and Theorem \ref{th5.4} yield the first statement. Moreover, combining the last statements of Theorems \ref{th4.17} and \ref{th5.2} one obtains other statements of the theorem.
\end{proof}
The following corollary is immediate from Theorem \ref{th4.18}.
\begin{corollary}\label{cor5.6}
If $\mul\Tmi=\{0\}$, then Theorems  \ref{th5.4} and \ref{th5.5} are valid for spectral functions $\Si(\cd)$ (instead of pseudospectral ones). Moreover, in this case equality \eqref{5.5} defines the operator $\wt T^\tau$ and the last statements of Theorem \ref{th5.5} hold with $V_\Si$ and $\wt T^\tau$ in place of $V_{0,\Si}$ and $\wt T_0^{\tau}$ respectively.
\end{corollary}

\end{document}